\documentclass[a4paper,12pt,draft, reqno]{amsart}

\usepackage[margin=30mm]{geometry}
\usepackage{amssymb, amsmath, amsthm, amscd}

\usepackage{eucal,mathrsfs,dsfont}
\usepackage{color}

\renewcommand{\leq}{\leqslant}
\renewcommand{\geq}{\geqslant}
\renewcommand{\le}{\leqslant}
\renewcommand{\ge}{\geqslant}

\definecolor{mno}{rgb}{0.5,0.1,0.5}

\newcommand{\R}{\mathbb R}
\newcommand{\bP}{\mathbb P}
\newcommand{\bE}{\mathbb E}
\def\bB {{\mathbb B}}

\newcommand{\sL}{\mathcal{L}}
\newcommand{\sE}{\mathcal{E}}
\newcommand{\sF}{\mathcal{F}}
\newcommand{\sD}{\mathcal{D}}
\newcommand{\sN}{\mathcal{N}}

\newcommand{\eps}{\varepsilon}

\newtheorem{theorem}{Theorem}[section]
\newtheorem{lemma}[theorem]{Lemma}
\newtheorem{proposition}[theorem]{Proposition}
\newtheorem{corollary}[theorem]{Corollary}
\numberwithin{equation}{section}

\theoremstyle{definition}
\newtheorem{definition}[theorem]{Definition}

\newtheorem{remark}[theorem]{Remark}

\begin{document}
\allowdisplaybreaks
\title[Heat kernel estimates for  time fractional equations]{\bfseries
Heat kernel estimates for  time fractional equations}

\author{Zhen-Qing Chen \quad Panki Kim \quad Takashi Kumagai
\quad Jian Wang}

\date{}
\maketitle

\begin{abstract}
In this paper, we establish existence and uniqueness of weak solutions to
general  time fractional equations and give their probabilistic
representations.
We then derive sharp
two-sided
estimates for fundamental solutions of
a family of
time fractional equations in metric measure spaces.

 \bigskip

\noindent
\textbf{Keywords:} Dirichlet form; subordinator; Caputo derivative;
heat kernel estimates;
time fractional  equation
\bigskip

\noindent \textbf{MSC 2010:}
 60G52; 60J25; 60J55;  60J35; 60J75.
\end{abstract}
\allowdisplaybreaks

\medskip

\section{Introduction}

\subsection{Motivation}

 Let $X=\{X_t, t\geq 0; \, \bP_x, x\in M\}$ be a strong Markov process on a separable locally compact Hausdorff space $M$
  whose transition semigroup $\{T_t, t\geq 0\}$ is a uniformly bounded strong continuous semigroup in some Banach space $(\bB, \| \cdot \|)$.
For example,     $\bB=L^p( M; \mu)$ for some measure $\mu$ with full support on $M$ and $p\geq 1$ or $\bB=
   C_\infty (M)$, the space of continuous functions  o   $M$ that vanish at infinity equipped with uniform norm.
  Let  $(\sL, \sD (\sL))$ be the infinitesimal generator of $\{T_t, t\geq 0\}$ in $\bB$.

Let $S=\{S_t: t\ge0\}$ be a subordinator (that is, a non-decreasing real valued L\'evy process
 with $S_0=0$)
without drift and having the Laplace exponent $\phi$:
$$
\bE  e^{-\lambda S_t}= e^{-t \phi(\lambda)} \quad \hbox{for all } t, \lambda>0 .
$$
It is well known
(see, e.g., \cite{SSV})
 that there exists a unique Borel measure $\nu$ on $(0,\infty)$ with $\int_0^\infty (1\wedge s)\,\nu(ds)<\infty$
such that
\begin{equation}\label{e:1.1}
\phi(\lambda)=\int_0^\infty (1-e^{-\lambda s})\, \nu(ds).
\end{equation}
The measure $\nu$ is called the L\'evy measure of the subordinator $S$.
Define for $t>0$, $E_t=\inf\{s>0: S_s>t\}$, the inverse subordinator.
We assume that $S$ is independent of $X$ and that  $\nu(0,\infty) =\infty$, excluding
compound Poisson processes.
Thus, almost surely, the function $t\mapsto S_t$ is strictly increasing, and hence $t\mapsto E_t$ is continuous.

 Recently, it is established as a particular case of
 \cite[Theorem 2.1]{Chen} that
for any $f\in \sD(\sL)$,
$$
u (t,x):= \bE \left[ T_{E_t}f(x) \right]=\bE^x \left[ f(X_{E_t}) \right]
$$
is the unique strong solution (in some suitable sense) to
the equation
\begin{equation}\label{e:1.3}
\partial_t^w u(t,x)=\sL u (t,x) \quad \hbox{with } u(0, x)=f(x),
\end{equation}
    where $w(s)=\nu(s,\infty)$ for $s>0$, and
$\partial_t^w$ is the
fractional
derivative defined as follows: for a function $\psi: [0, \infty)\to \R$,
\begin{equation}\label{e:1.4}
\partial_t^w  \psi (t) :=\frac{d}{dt}\int_0^t w(t-s) (\psi (s)-\psi(0))\,ds.
\end{equation}
 See Theorem \ref{T:2.1} in Section \ref{S:2} for a precise statement.
  If the semigroup $\{T_t:t\ge0\}$
    (or equivalently, the strong Markov process $X$) has a
       density kernel $q(t, x, y)$
  with respect to  a measure $\mu$ on $M$,
then by  Fubini's theorem,  for any bounded function $f\in  \sD(\sL)$,
\begin{align} \label{e:sou}
u(t,x) &= \bE \left[ T_{E_t}f(x) \right]=\int_0^\infty T_r f(x)\,d_r\bP (E_t\le r)=\int_0^\infty T_r f(x)\, d_r\bP (S_r\ge t) \\
&= \int_0^\infty \int_M f(y)q(r,x,y)\,\mu(dy)\, d_r\bP (S_r\ge t) \nonumber \\
&=\int_M f(y) \left( \int_0^\infty q(r,x,y)\, d_r\bP (S_r\ge t)\right) \mu(dy). \nonumber
\end{align}
Here and in what follows, $d_r$ denotes the (generalized) derivative with respect to $r$.
This says that
\begin{equation} \label{e:1.5}
 p(t,x,y):=\int_0^\infty q(r,x,y)\, d_r\bP (S_r\ge t)
 \end{equation}
is the ``fundamental solution" to the time fractional equation \eqref{e:1.3}.
Note that in the PDE literatures, the most standard approach to analyze $p(t,x,y)$ is
to use the Mittag-Leffler function, and then take the inverse Fourier
transform (see for instance \cite{EK}, about detailed estimates of $p(t,x,y)$ when $\{S_t\}$ is a
$\beta$-stable subordinator).
We emphasize
that the expression \eqref{e:1.5} is more intuitive, simple, and general
(in the sense that we do not rely on the Fourier transform).

When $S=\{S_t: t\geq 0\}$ is a
$\beta$-stable subordinator
with the Laplace exponent $\phi (\lambda) = \lambda^\beta$ for some $0<\beta<1$, $S$ has no drift
and its L\'evy measure is given by
$\nu (ds)= \frac{\beta}{\Gamma (1-\beta)} s^{-(1+\beta)} \,ds$.
In this case
$$
w(s)= \nu (s, \infty)= \int_s^\infty \frac{\beta}{\Gamma (1-\beta)} y^{-(1+\beta)} \,dy = \frac{s^{-\beta}}{\Gamma (1-\beta)},
$$
and so the time fractional derivative $\partial^w_t f $ defined by \eqref{e:1.4} is just the Caputo derivative of order $\beta$
in literature.

The time fractional diffusion equation
\eqref{e:1.3} with $\sL =\Delta$
has been widely used
to model anomalous diffusions  exhibiting  subdiffusive behavior,
due to particle sticking and trapping phenomena (see e.g. \cite{MSi,SKW82}).
It can be used to model ``ultraslow diffusion"  where a plume spreads at a logarithmic rate,
for example when $S$ is a subordinator of
mixed stable subordinators;
see \cite{MS1} for details.
The time fractional diffusion equation also appears as
 a scaling limit of random walk on $\mathbb Z^d$
with
heavy-tailed random conductance:
Let $\{C_{xy}: x,y\in \mathbb Z^d, |x-y|=1\}$ be positive
i.i.d. random variables such that
$C_{xy}=C_{yx}$,  ${\mathbb P}(C_{xy}\ge 1)=1$ and
\[
{\mathbb P}(C_{xy}\ge u)=c_1u^{-\alpha}(1+o(1))\quad \mbox{ as }u\to\infty
\]
for some constants $c_1>0$ and $\alpha\in (0,1)$. Let $\{Y_t\}_{t\ge 0}$ be the Markov chain whose transition probability from
$x$ to $y$ is equal to $C_{xy}/\sum_{z\in \mathbb Z^d} C_{xz}$. Then, for $d\ge 3$, $\{\eps Y_{t\eps^{-2/\alpha}}\}_{t\ge0}$ converges to a multiple of
the Caputo time fractional
 diffusion process on the path space equipped with the Skorokhod $J_1$-topology ${\mathbb P}$-almost
surely as $\eps\to 0$; see \cite{BarC}. For $d=2$, the same result holds by changing the scaling as $\{\eps Y_{t(\log (1/\eps))^{1-1/\alpha}\eps^{-2/\alpha}}\}_{t\ge0}$; see
\cite{Cer}.

Time fractional diffusion equations have possible applications to anomalous diffusions in soil; see for instance \cite{NSY}. The ultimate goal in the application is to determine the microstructure of soil
through the averaged spatial data analysis, and to predict the progress of soil contamination. For such analysis, there
is no reason that the operator in the
master equation \eqref{e:1.3}
is the classical Laplace operator in Euclidean space, and it would be useful to consider more general operators in metric measure spaces (\cite{Nak}).
 In fact there are literatures that discuss the time fractional equation \eqref{e:1.3} in which
$\sL$ is a fractional Laplacian; see \cite{BHG,SZ,Zas}. In \cite{BHG} the authors discuss applications to laws of human travels, and in \cite{SZ,Zas} applications to chaotic Hamiltonian dynamics are discussed in typical low dimensional systems.
Therefore, it is interesting and desirable to obtain explicit two-sided estimates of $p(t, x, y)$ for more general operators in non-Euclidean spaces.

The goal of this paper is to accomplish
 this, assuming general apriori estimates (see \eqref{e:hkmequi} and \eqref{eq:fibie3} below)
  for the fundamental solution of the heat equation of
the infinitesimal spatial generator $\sL$, and some weak scaling property on the subordinator $S$ (see \eqref{e:phi}).
Moreover, when $X$ is a symmetric Markov process with respect to some measure $\mu$ on $M$,
we will show in Section \ref{S:2} that for every $f\in L^2(M; \mu)$,
 $u (t,x):= \bE \left[ T_{E_t}f(x) \right]$ is the unique weak solution to \eqref{e:1.3};
 see Theorem \ref{T:2.4} for details.

\medskip

In what follows, we write
$h(s)\simeq f(s)$
if there exist constants $c_{1},c_{2}>0$ such that
$
c_{1}f(s)\leq h(s)\leq c_{2}f(s),
$
for the specified range of the argument $s$. Similarly, we write
$h(s)\asymp f(s)g(s)$
if there exist constants $C_{1},c_{1},C_{2},c_{2}>0$ such that
$
f(C_{1}s)g(c_{1}s)\leq h(s)\leq f(C_{2}s)g(c_{2}s)
$
for the specified range of $s$.
 $c$ (without subscripts) denotes a strictly positive
constant  whose value
is unimportant and which  may change from line to line.
Constants $c_0, c_1, c_2, \ldots$ with subscripts denote strictly positive
constants and the labeling of the constants $c_0, c_1, c_2, \ldots$ starts anew in
the statement of each result and the each step of its proof.
We will use ``$:=$" to denote a definition,
which is read as ``is defined to be".
 For any $a, b\in \R$, we use the notations
$a\wedge b:=\min \{a, b\}$ and $a\vee b:=\max\{a, b\}$.
Sometimes we use the notation $\frac{\partial v(t, r)}{\partial t}=\partial_t v(t, r)$.

\subsection{Special case: self-adjoint spatial generators on $d$-sets
}

Before giving our main results in full generality, we first give a version of them which
can be described in a tidy way.

Throughout  this paper except Section \ref{S:2},
  $(M, d)$ is  a locally compact separable metric space and $\mu$ is
 a Radon measure on $( M,d )$ that has full support.
We say that the metric space $(M,d)$ satisfies the \emph{chain condition}
if there exists a constant $C>0$ such that, for any $x,y\in
M$ and for any $n\in {\mathbb N}$, there exists a sequence
$\{x_{i}\}_{i=0}^{n}\subset M$ such that $x_{0}=x$, $x_{n}=y$, and
\begin{equation*}
d(x_{i},x_{i+1})\leq C \frac{ d(x,y)}{n}\qquad \text{for all }i=0,1,\cdots,n-1.
\end{equation*}

Suppose that $X$ is an $\mu$-symmetric Hunt process associated with a
regular Dirichlet form
$(\mathcal{E}, \mathcal F)$ on $L^2(M; \mu)$
and it has
a transition density function $q(t, x, y)$ with respect to the measure
$\mu$. We call $q(t, x, y)$ the heat kernel of
$X$.
Suppose that the heat kernel enjoys the
following estimates
\begin{equation}\label{e:hk}
q(t,x,y)\asymp \frac{1}{t^{d/\alpha}}F\left(\frac{ \,d(x,y)}{t^{1/\alpha}}\right),\quad t>0, x,y\in M,
\end{equation}
where $d,\alpha>0$ and $F :[0,+\infty )\rightarrow \lbrack 0,+\infty )$ is
 a non-increasing
 function  such that $F (s_0) >0$ for some $s_0>0$.
In \cite[Theorem 4.1]{GK}, it is proved that if
$(M,d)$ satisfies the chain condition and
$(\mathcal{E}, \mathcal F)$ is conservative,
then there are only two possible shapes of $F$.

\begin{theorem}{\rm (\cite[Theorem 4.1]{GK})} \label{dichoTh}
Assume that the metric space $( M,d ) $ satisfies
the chain condition and all balls are relatively compact.
Assume further that $(\mathcal{E}, \mathcal F)$ is regular, conservative and
\eqref{e:hk} holds with some $d,\alpha>0$ and
non-increasing
 function $F$.
Then $\alpha \leq d +1$,
$\mu \left( B\left( x,r\right) \right) \simeq r^{d }$ for all $x\in M$ and $r>0$,
and the following dichotomy holds:
either the Dirichlet form
 $(\sE, \sF)$ is local, $\alpha \geq 2$,
$M$ is connected, and
\begin{equation*}
F \left( s\right) \asymp \exp \left( -s^{\alpha / (\alpha -1) }\right),
\end{equation*}
or the Dirichlet form
 $(\sE, \sF)$ is
of pure jump type and
\begin{equation*}
F \left( s\right)\simeq \left( 1+s\right) ^{-\left( d+\alpha
\right) }.
\end{equation*}
\end{theorem}

\medskip

In other words,  Theorem \ref{dichoTh} assets that under assumptions in the theorem,
$(M, d , \mu)$ is an Alfhors $d$-regular set and the heat kernel $q(t, x, y)$ has the following estimates:
\begin{equation}\label{e:1.7}
q(t, x, y) \asymp t^{-d/\alpha} \exp \left( - \left(\frac{  d(x, y)^\alpha }{ t} \right)^{1/(\alpha -1)}  \right),
\quad t>0, x, y\in M
\end{equation}
 for some $\alpha \geq 2$ when $(\sE, \sF)$ is local
and $M$ is connected, or
 \begin{equation}\label{e:1.8}
q(t, x, y) \simeq t^{-d/\alpha} \left( 1+ \frac{ \,d(x,y)}{t^{1/\alpha}}  \right)^{- ( d+\alpha) }
\simeq t^{-d/\alpha} \wedge \frac t{d(x,y)^{d+\alpha}},\quad t>0, x, y\in M
\end{equation}
for some $\alpha >0$ when $ (\sE, \sF)$ is
of pure jump type.
Property \eqref{e:1.7}  is called the
sub-Gaussian heat kernel estimates, and \eqref{e:1.8} is called the
$\alpha$-stable-like heat kernel estimates.
\ \

\begin{definition}\label{theorem:defwsc}
Suppose that $0<\alpha_1 \le \alpha_2 < \infty$.
We say that a non-decreasing function $\Psi: (0,\infty)\to (0,\infty)$ satisfies the \emph{weak scaling property with $(\alpha_1,\alpha_2)$}
if there exist constants $c_1,c_2>0$
such that \begin{equation}\label{e:lleeqpa}
c_1(R/r)^{\alpha_1}
\le\Psi(R)/\Psi(r) \le c_2(R/r)^{\alpha_2}
\quad \text{for all } 0<r \le R < \infty.
\end{equation}
We say that a family of non-decreasing functions $\{\Psi_x\}_{x\in \Lambda}$  satisfies the \emph{weak scaling property uniformly with $(\alpha_1,\alpha_2)$} if each $\Psi_x$ satisfies the weak scaling property with constants
$c_1,c_2>0$ and $0<\alpha_1 \le \alpha_2 < \infty$ independent of the choice of $x\in \Lambda$.
\end{definition}

\medskip

Throughout the paper, we assume that the Laplace exponent $\phi$ of the driftless subordinator $S=\{S_t: t\geq 0\}$ satisfies the weak scaling property with $(\beta_1,\beta_2)$ such that $0<\beta_1\le \beta_2<1$; namely,
for any $\lambda>0$ and $\kappa\ge 1$,
\begin{equation}\label{e:phi} c_1 \kappa^{\beta_1}\le \frac{\phi(\kappa \lambda)}{\phi(\lambda)}\le c_2 \kappa^{\beta_2}.\end{equation}
Note that under \eqref{e:phi}, the L\'evy measure $\nu$ of $S$ is infinite as
  $\nu(0,\infty)=\lim\limits_{\lambda \to \infty}\phi(\lambda)=\infty$,
 excluding
 compound Poisson processes.

\ \

The following is the main result in this subsection on the two-sided sharp estimates  for the fundamental solution $p(t,x,y)$ of the time fractional equation \eqref{e:1.3}.

\begin{theorem}\label{T:1.2}
 Assume conditions in Theorem $\ref{dichoTh}$ and \eqref{e:phi} hold. Let $p(t,x,y)$ be given by \eqref{e:1.5}. Then, we have
\begin{itemize}
\item[(i)] If $d(x,y) \phi(t^{-1})^{1/\alpha}\le 1$, then
\begin{align*}
p(t,x,y) \simeq
&\begin{cases}
 \phi(t^{-1})^{d/\alpha}  &  \hbox{if }  d<\alpha, \\
 \phi(t^{-1})
\displaystyle \log \left(\frac{2}{d(x,y)\phi(t^{-1})^{1/\alpha}}\right)  &   \hbox{if } d=\alpha,\\
\phi(t^{-1})^{d/\alpha} \left(d(x,y) \phi(t^{-1})^{1/\alpha}\right)^{-d+\alpha}
=\phi(t^{-1})/ d(x,y)^{d-\alpha}
  & \hbox{if }   d>\alpha.
\end{cases}
\end{align*}

\item[(ii)] Suppose  $d(x,y) \phi(t^{-1})^{1/\alpha}\ge 1$.
When the Dirichlet form $(\sE, \sF)$ is local,
\begin{equation}\label{e:08-0} p(t,x,y)\asymp \phi(t^{-1})^{d/\alpha} \exp\Big(- t{\bar\phi_\alpha}^{-1}((d(x,y)/t)^\alpha)\Big),
\end{equation}
where $\displaystyle \bar\phi_\alpha(\lambda)= \lambda^\alpha / {\phi(\lambda)}$,
and $ \bar\phi_\alpha^{-1} (\lambda)$ is the inverse function of  $\bar\phi_\alpha(\lambda)$,
i.e.,\ $\bar\phi_\alpha^{-1} (\lambda):=\inf\{s>0: \bar\phi_\alpha(s)\ge \lambda\}$ for all $\lambda\ge0$;
when  $(\sE, \sF)$ is
of pure jump type,
$$p(t,x,y)\simeq
  \phi(t^{-1})^{d/\alpha} (d(x,y)\phi(t^{-1})^{1/\alpha})^{-d-\alpha}
=\frac1{ \phi(t^{-1}) d(x, y)^{d+\alpha} }.
$$
\end{itemize}
\end{theorem}

\medskip

\begin{remark} \label{e:remark}
At first glance, the estimate \eqref{e:08-0} may look odd  since the term $d(x,y)/t$ appears instead of the scaling term $d(x,y)\phi(t^{-1})^{1/\alpha}$ which appears
in the rest of the estimates in Theorem \ref{T:1.2}.
However, since
$$
t{\bar\phi_\alpha}^{-1}((d(x,y)/t)^\alpha)=\frac{{\bar\phi_\alpha}^{-1}(d(x,y)^\alpha/t^\alpha)}{{\bar\phi_\alpha}^{-1}(\bar\phi_\alpha(t^{-1}))}=
\frac{{\bar\phi_\alpha}^{-1}(d(x,y)^\alpha/t^\alpha)}{{\bar\phi_\alpha}^{-1}(1/(\phi(t^{-1})t^\alpha))},
$$
they are consistent.
 \end{remark}

\medskip

Let us consider a special case of
Theorem \ref{T:1.2} where $\{S_t:t\ge0\}$ is a $\beta$-stable subordinator for some $\beta \in (0, 1)$. In this case,  $\phi(s)=s^\beta$. Define
\begin{align*}
H_{\le 1}(t,d(x,y)) =&\begin{cases}
 t^{-\beta d/\alpha}, &   d<\alpha, \\
 t^{-\beta}\displaystyle \log \left(\frac{2}{{d(x,y)}t^{-\beta /\alpha}}\right),
 &  d=\alpha,\\
t^{-\beta d/\alpha} \left(d(x,y)t^{-\beta/\alpha}\right)^{-d+\alpha}
=t^{-\beta}/d(x,y)^{d-\alpha}, &  d>\alpha,
\end{cases}\\
H_{\ge 1}^{(c)}(t,d(x,y))
=& t^{-\beta d/\alpha}\exp\Big((d(x,y)t^{-\beta/\alpha})^{\alpha/(\alpha-\beta)}\Big),\\
H_{\ge 1}^{(j)}(t,d(x,y))
=& t^{-\beta d/\alpha}(d(x,y)t^{-\beta/\alpha})^{-(d+\alpha)}
=t^\beta/d(x,y)^{d+\alpha}.
\end{align*}

\medskip

 \begin{corollary}\label{T:special-case}
  Assume that conditions in Theorem $\ref{dichoTh}$ hold and $\phi(s)=s^\beta$ for $0<\beta<1$. Let $p(t,x,y)$ be given by \eqref{e:1.5}.
  \begin{itemize}
 \item[(i)] Suppose $F(s)=\exp(-s^{\alpha/(\alpha-1)})$ with $\alpha\ge 2$. Then
 \begin{align*}
p(t,x,y) \simeq H_{\le 1}(t,d(x,y))&\quad\mbox{ if }  d(x,y) t^{-\beta /\alpha}\le 1,\\
p(t,x,y) \asymp
H_{\ge 1}^{(c)}
(t,d(x,y))&\quad\mbox{ if }  d(x,y) t^{-\beta /\alpha}\ge 1.
 \end{align*}
\item[(ii)] Suppose $F(s)=(1+s)^{-d-\alpha}$. Then,
 \begin{align*}
p(t,x,y) \simeq H_{\le 1}(t,d(x,y))&\quad\mbox{ if }  d(x,y) t^{-\beta /\alpha}\le 1,\\
p(t,x,y) \simeq
H_{\ge 1}^{(j)}(t,d(x,y))
&\quad\mbox{ if }  d(x,y) t^{-\beta /\alpha}\ge 1.
 \end{align*}\end{itemize}
\end{corollary}

We note that when $M=\R^d$ and $X$ is
a  rotationally symmetric $\alpha$-stable process on $\R^d$,
 (part of statements in) Corollary \ref{T:special-case} (ii) have been obtained in \cite[Lemma 2.1]{FN}.

\subsection{General case}
In this subsection, we give
a general version
of the heat kernel estimates for the time fractional equation \eqref{e:1.3}.

Recall that  $( M,d, \mu ) $ is
a locally compact separable metric measure space such that $\mu$ is a Radon measure on
$( M,d )$  with full support.
Throughout this paper
we assume
$X$ is a strong Markov process on $M$ having infinite lifetime.
 For $x\in M$ and $r\ge 0$, define
\[
V(x,r)=\mu(B(x,r)).
\]
We further assume that for each $x\in M$, $V(x,\cdot)$ satisfies the weak scaling property
uniformly with $(d_1,d_2)$ for some $d_2 \ge d_1>0$;
that is, for any $0<r\le R$ and $x\in M$,
\begin{equation}\label{vd1} c_1\left(\frac R r\right)^{d_1}\le \frac{V(x,R)}{V(x,r)}\le c_2 \left( \frac R r\right)^{d_2}.\end{equation}
Note that \eqref{vd1} is equivalent to the so-called volume doubling and reverse volume doubling
conditions.
As in the previous section, we also assume that the Laplace exponent $\phi$ of the driftless subordinator $S=\{S_t: t\geq 0\}$ satisfies \eqref{e:phi}.

\subsubsection{{\bf Pure jump case}}
We first consider the case that the
strong Markov process $X$ has a transition density function $q(t, x, y)$ with respect to $\mu$
 that enjoys the following two-sided estimates:
\begin{equation}\label{e:hkmequi}q(t,x,y)\simeq
\frac{1}{V(x, \Phi^{-1}(t))} \wedge \frac{t}{V(x, d(x,y))\Phi(d(x,y))},
\quad t>0, x,y\in M, \end{equation}
where $\Phi :[0,+\infty )\rightarrow \lbrack 0,+\infty )$
is a strictly increasing function with $\Phi(0)=0$ that satisfies the weak scaling property with $(\alpha_1,\alpha_2)$, i.e., \eqref{e:lleeqpa} is satisfied.

Examples of symmetric Markov processes satisfying the above condition can be found in \cite{CK,CKW}.
It is known that symmetric Markov processes  enjoy heat kernel estimates \eqref{e:hkmequi} are of pure jump type.
Note that when $V(x,r)\simeq r^d$ and $\Phi(s)=s^\alpha$ for $r,s>0$ and $x\in M$, then for any $x,y\in M$ and $t>0$,
$V(x, \Phi^{-1}(t))\simeq t^{d/\alpha}$ and $V(x, d(x,y))\Phi(d(x,y))
\simeq d(x,y)^{d+\alpha}$, so
\eqref{e:hkmequi} boils down to \eqref{e:1.8}.

\medskip

Here is the heat kernel estimates for the time fractional equation \eqref{e:1.3}.

\begin{theorem}\label{theorem:mainjump}
Suppose that the heat kernel of
the strong Markov process $X$ enjoys two-sided estimates
\eqref{e:hkmequi}. Let $p(t,x,y)$ be given by \eqref{e:1.5}. Then we have the following two statements:
\begin{itemize}
\item[(i)]
If   $\Phi(d(x,y))\phi(t^{-1}) \le 1$, then
\begin{equation}\label{e:llffoo}\begin{split} p(t,x,y) &\simeq   \phi(t^{-1}) \int_{\Phi(d(x,y))}^{2/\phi(t^{-1})} \frac{1}{V(x,\Phi^{-1}(r))}\,dr\\
    &=\int_{\Phi(d(x,y))\phi(t^{-1})}^{2} \frac{1}{V(x,\Phi^{-1}(r/\phi(t^{-1})))}\,dr.
		\end{split}
		\end{equation}
		
\item[(ii)]
If   $\Phi(d(x,y))\phi(t^{-1}) \ge 1$, then $$p(t,x,y)\simeq
  \frac1{\phi(t^{-1}) V(x,d(x,y))\,  \Phi(d(x,y)) }.
 $$
\end{itemize}
\end{theorem}

\begin{remark}\label{rem-uplow2} (1) Note that, by some elementary calculations (see \eqref{e:rrffee1} and \eqref{e:com1} below), we have
$$\left(\frac{1}{V(x, \Phi^{-1}(1/\phi(t^{-1})))}\vee  \frac{\Phi(d(x,y))\phi(t^{-1}) }{V(x,d(x,y))}\right)\le c \phi(t^{-1}) \int_{\Phi(d(x,y))}^{2/\phi(t^{-1})} \frac{1}{V(x,\Phi^{-1}(r))}\,dr.$$
Roughly speaking,
when $s\mapsto s^{-1} V(x,\Phi^{-1}(s))$ is  strictly increasing,
$$
p(t,x,y) \simeq \frac{\Phi(d(x,y))\phi(t^{-1}) }{V(x,d(x,y))}\quad \hbox{if } \Phi(d(x,y))\phi(t^{-1}) \le 1.
$$
When $s\mapsto s^{-1} V(x,\Phi^{-1}(s))$ is strictly deceasing,
$$
p(t,x,y) \simeq \frac{1}{V(x, \Phi^{-1}(1/\phi(t^{-1})))}\quad \hbox{if } \Phi(d(x,y))\phi(t^{-1}) \le 1.
$$
 For the critical case, the logarithmic factor will appear, see Theorem \ref{T:1.2} (i) and Corollary \ref{T:special-case}, or Corollary \ref{T:mix2} for the explicit
 statements.

(2)
As the proof shows, if we assume the upper (resp.\ lower) bound in \eqref{e:hkmequi}, then
the upper (resp.\ lower) bounds of $p(t,x,y)$ hold in the statement of Theorem \ref{theorem:mainjump}.
\end{remark}
\medskip

\subsubsection{{\bf Diffusion case}}

We next consider the case that
the strong Markov process $X$ is a conservative diffusion;
that is, it has continuous sample paths and infinite lifetime.
In this case,
we further assume that
the metric space $( M,d ) $ is connected and satisfies the chain condition.
Moreover, we assume that the heat kernel of the diffusion $X$ with respect to $\mu$ exists and  enjoys the
following two-sided estimates
\begin{equation}\label{eq:fibie3}
q(t,x,y)\asymp \frac{1}{V(x, \Phi^{-1}(t))} \exp\left(-m(t,d(x,y))\right),
\quad t>0, x, y\in M.
\end{equation}
Here, $\Phi :[0,+\infty )\rightarrow \lbrack 0,+\infty )$
is a strictly increasing function with $\Phi(0)=0$, and satisfies the weak scaling property with $(\alpha_1,\alpha_2)$ such that
the constants $\alpha_2 \ge \alpha_1>1$ in \eqref{e:lleeqpa};
the function $m(t,r)$
is strictly positive for all $t,r>0$,
 non-increasing  on $(0,\infty)$ for fixed $r>0$, and  determined by
\begin{equation}\label{e:scdf}
\frac{t}{m(t,r)}\simeq \Phi\left(\frac{r}{m(t,r)}\right),\quad t,r>0.\end{equation} In particular, by \eqref{e:lleeqpa} with $\alpha_1>1$ and \eqref{e:scdf}, there are constants $c_1,c_2>0$ such that for all $r>0$,
\begin{equation}\label{e:scdf-1}c_1\left(\frac{T}{t}\right)^{-1/(\alpha_1-1)}\le \frac{m(T,r)}{m(t,r)}\le c_2\left(\frac{T}{t}\right)^{-1/(\alpha_2-1)},\quad 0<t\le T.
\end{equation}
On the other hand, by \eqref{e:scdf} we have
 \begin{align}
 \label{e:mphisim1}
 m(\Phi(r),r)\simeq1, \quad  r>0.
 \end{align}
Using this and the fact that  $m(\cdot,r)$ is non-increasing,
we have
\begin{equation}\label{d:ond}
q(t,x,y)\simeq  \frac{1}{V(x, \Phi^{-1}(t))}
\quad \hbox{when } \Phi(d(x,y))\le c_3 t.
\end{equation}

Note that when $V(x,r)\simeq r^d$ and $\Phi(s)=s^\alpha$ for $r,s>0$ and $x\in M$, then for every $x,y\in M$ and $t>0$,
$V(x, \Phi^{-1}(t))\simeq t^{d/\alpha}$ and $m(t,d(x,y))\simeq (d(x,y)^\alpha/t)^{1/(\alpha-1)}$,
and so \eqref{eq:fibie3} is reduced to \eqref{e:1.7}.
Examples of
conservative symmetric diffusions satisfying  condition \eqref{eq:fibie3}
include diffusions on fractals such as Sierpinski gaskets and Sierpinski carpets.
For example, Brownian motion on the $2$-dimensional Sierpinski gasket enjoys
\eqref{e:1.7} (hence \eqref{eq:fibie3}) with $d=\log 3/\log 2$ and $\alpha=\log 5/\log 2>2$.
See \cite{HK} and \cite[Section 13]{telcs} for more examples.
Intuitively, $m(t,d(x,y))$ in \eqref{eq:fibie3} is an optimal number of
steps for diffusions to reach from $x$ to $y$
at time $t$. As one sees in \eqref{e:scdf}, the time and the distance are divided by $m(t,d(x,y))$ so that the
relation between them is given by $\Phi$. Then one decomposes the path from $x$ to $y$ into
$m(t,d(x,y))$-th \lq most probable\rq\, paths on which the near-diagonal heat kernel estimates hold,
and uses the chain argument.
This is how the off-diagonal estimates (exponential part of \eqref{eq:fibie3}) can be deduced on various
concrete examples such as diffusions on fractals.

\medskip

Here is the heat kernel estimates for the time fractional equation \eqref{e:1.3}.
\begin{theorem}\label{theorem:maindiff}
Suppose that the heat kernel of
the conservative diffusion process $X$
enjoys estimates
\eqref{eq:fibie3}. Let $p(t,x,y)$ be given by \eqref{e:1.5}. Then we have the following two statements:
\begin{itemize}
\item[(i)] If $\Phi(d(x,y))\phi(t^{-1})\le 1$, then
\begin{align*}p(t,x,y) \simeq &\int_{\Phi(d(x,y))\phi(t^{-1})}^{2}\frac{1}{V(x,\Phi^{-1}(r/\phi(t^{-1})))}\,dr.\end{align*}

\item[(ii)]If $\Phi(d(x,y))\phi(t^{-1})\ge 1$, then
there exist constants $c_i>0$ $(i=1, \dots, 4)$ such that
\begin{align*}
\frac{c_1}{V(x, \Phi^{-1}(1/\phi(t^{-1}))))} & \exp(-c_2n(t,d(x,y)))\\
  & \le   p(t,x, y) \\
  &\le \frac{c_3}{V(x, \Phi^{-1}(1/\phi(t^{-1}))))} \exp(-c_4n(t,d(x,y))),
\end{align*}
where
$n(\cdot, r)$ is a non-increasing function on $(0,\infty)$ determined by
\begin{align}
\label{e:phnd}
\frac{1}{\phi(n(t,r)/t)} \simeq \Phi\left(\frac{r}{n(t,r)}\right),\quad t,r>0.
\end{align}
\end{itemize}
\end{theorem}
\begin{remark}\label{rem-uplow}
(1)
As mentioned above, $p(t,x,y)$ given by \eqref{e:1.5} is the ``fundamental solution" to the time fractional equation \eqref{e:1.3}, and so $p(t,x,y)$ closely relates to the process $X_E:=\{X_{E_t}: t\ge0\}$, where $\{E_t:t\ge0\}$ is
the inverse subordinator with respect to $S$.
 Estimates for the distribution of subordinator $S$ collected in Proposition \ref{p:sub}\,(i) below show that, from the process $X$ to the time-change process $X_E$, the time scale will be changed from $t$ to $1/\phi(t^{-1})$. By this observation, we can partly give the intuitive explanation of the shape of the heat kernel estimates in
Theorems \ref{theorem:mainjump} and \ref{theorem:maindiff}. In particular, the case
that  $\Phi(d(x,y))\phi(t^{-1})\le 1$ corresponds to ``near-diagonal" estimates of $p(t,x,y)$, while the case that $\Phi(d(x,y))\phi(t^{-1})\ge 1$ can be regarded as ``off-diagonal" estimates.

(2) When $\Phi(d(x,y))\phi(t^{-1})\le 1$, two-sided estimates of $p(t,x,y)$ for  time fractional diffusion  processes
enjoy the same form as these for time fractional jump processes,
see Theorems \ref{theorem:mainjump} (i) and \ref{theorem:maindiff} (i).
Similar to Theorem \ref{theorem:mainjump},
as the proof shows, if we assume the upper (resp.\ lower) bound in \eqref{eq:fibie3}, then
the upper (resp.\ lower) bounds of $p(t,x,y)$ hold in the statement of Theorem \ref{theorem:maindiff}.
\end{remark}

\bigskip

The rest of the paper is organized as follows.
In Section \ref{S:2}, we assume $X$ is an $\mu$-symmetric strong Markov process on $M$.
Then its transition semigroup $\{T_t; t\geq 0\}$ is a strongly continuous contraction semigroup in $L^2(M; \mu)$.
Denote its infinitesimal generator by $(\sL, \sD (\sL))$. We show that
for every $f\in L^2(M; \mu)$, the general time fractional equation  \eqref{e:1.3} has a unique
weak solution $u(t, x)$ in $L^2(M; \mu)$ with initial value $f$, and the solution has a representation
$u(t, x)=\bE \left[ f (X_{E_t}) \right]$. This result relaxes the condition
that $f\in \sD (\sL)$ imposed in \cite[Theorem 2.3]{Chen} at the expense of formulating
the solution to \eqref{e:1.3} in the weak sense rather than in the strong sense
and by imposing symmetry on the process $X$.
In Section \ref{S:3}, we present some preliminary estimates about Bernstein functions and subordinators. In particular, we establish the relation between the weak scaling property and Bernstein functions, which is interesting of its own.
Section \ref{S:4} and Section \ref{S:5} are devoted to proofs of
the main results of this paper, Theorems \ref{theorem:mainjump} and \ref{theorem:maindiff}, respectively.
Theorem \ref{T:1.2} is then obtained as a corollary of Theorems
\ref{theorem:mainjump} and \ref{theorem:maindiff}.

\section{Time fractional equations}\label{S:2}

 Let $X=\{X_t, t\geq 0; \, \bP_x, x\in  M\}$ be a strong Markov process on
  a separable locally compact Hausdorff space
 $M$
  whose transition semigroup $\{T_t, t\geq 0\}$ is a uniformly bounded strong continuous semigroup in some Banach space $(\bB, \| \cdot \|)$.
   Let  $(\sL, \sD (\sL))$ be the infinitesimal generator of $\{P_t, t\geq 0\}$ in $\bB$.
Recall that $S=\{ S_t:  t\ge 0 \}$ is a subordinator with the Laplace exponent $\phi$
given by \eqref{e:1.1} with the infinite L\'evy measure $\nu$.
Define $w(x)=\nu(x,\infty)$ for $x>0$. Since $\nu(0,\infty)=\infty$, almost surely, $t\mapsto S_t$ is strictly increasing.
The following is a particular case of  a recent result established in
\cite[Theorem 2.1]{Chen}.

\begin{theorem}\label{T:2.1}
{\rm (\cite[Theorem 2.1]{Chen})}
 For  every $f\in \sD (\sL)$, $u(t, x):= \bE \left[ T_{E_t} f (x) \right]$ is a solution in   $(\bB, \| \cdot \|)$ to
 the time fractional equation \eqref{e:1.3} in the following sense:

\begin{itemize}

\item[\rm (i)]   $x\mapsto u(t, x)  $ is in $\sD (\sL)$ for each $t\geq 0$,
  and  both $t\mapsto u(t, \cdot)$ and $t\mapsto \sL u(t, \cdot)$
  are continuous and bounded in $(\bB, \| \cdot \|)$.
  Consequently,
  $$
  I^w_t (u (\cdot, x))  :=\int_0^t w(t-s) (u(s, x) -f(x))\, ds
  $$
   is absolutely convergent in
$(\bB, \| \cdot \|)$ for every $t>0$.

\item[\rm (ii)] For every $t>0$,  $$
  \lim_{\delta \to 0} \frac{1}{\delta}  \left(  I_{t+\delta}^w(u (\cdot, x)) - I_t^w(u (\cdot, x)) \right) = \sL u(t, x) \quad \hbox{in }
(\bB, \| \cdot \|).
$$
\end{itemize}

Conversely, if $u(t, x)$ is a solution to \eqref{e:1.3} in the sense of  {\rm (i)}  and {\rm (ii)}  above with  $f\in \sD (\sL)$,
 then $u(t, x)= \bE \left[ T_{E_t} f (x) \right]$ in
 $\bB$
for every $t\geq 0$.
\end{theorem}

\medskip

When $S=\{S_t: t\geq 0\}$ is a  $\beta$-stable subordinator with the
Laplace exponent $\phi (\lambda) = \lambda^\beta$ for $0<\beta<1$,
 its L\'evy measure  $\nu (dx)= \frac{\beta}{\Gamma (1-\beta)} x^{-(1+\beta)} \,dx$ and so
 $ w(x)= \mu (x, \infty)=   \frac{x^{-\beta}}{\Gamma (1-\beta)}$.
Hence   Theorem \ref{T:2.1} recovers the main result of \cite{BM} and \cite[Theorem 5.1]{MS0}
  for  parabolic equations with Caputo time derivative of order $\beta$.
For other related results, see
\cite[Remark 2.1]{Chen}.

\medskip

In this section, we show
that, when $X$ is $\mu$-symmetric for some $\sigma$-finite measure
$\mu$ with full support on $M$,  the initial condition $f\in \sD (\sL)$
in Theorem \ref{T:2.1} with $\bB:=L^2(M; \mu)$ can
 be weakened to $f\in L^2 (M; \mu)$
if we formulate the solution to the time fractional equation \eqref{e:1.3} in weak sense.
In this case, $(\sL, \sD(\sL))$ is the infinitesimal generator of the transition semigroup $\{T_t; t\geq 0\}$
in $L^2(M; \mu)$, which is a strong continuous contraction semigroup.
Denote by $(\sE, \sF)$ the Dirichlet form associated with $X$ in $L^2(M; \mu)$, which is known to be
quasi-regular; see \cite{CF} for example.

\medskip

First we recall the following result from
  \cite[Lemma 2.1 and Corollary 2.2 (i)]{Chen}.

\begin{lemma}\label{L:solu-1}
There is a Borel set $\sN \subset (0, \infty)$ having zero Lebesgue measure so that
  $$
 \bP (  S_s\geq  t)=\int_0^s \bE \left[ w(t-  S_r) {\bf1}_{\{t\geq   S_r \}}\right]  dr
 \quad \hbox{for every } s>0 \hbox{ and } t\in (0, \infty) \setminus \sN
 $$
and
$$ \int_0^\infty \bE \left[ w(t-  S_r) {\bf1}_{\{t\geq  S_r \}}\right]  dr =1
\quad \hbox{for every } t\in (0, \infty)\setminus \sN .
$$
 \end{lemma}

Define $G(0)=0$ and $G(x)=\int_0^x w(t)\,dt$ for all $x>0$. We also need the following lemma, which is  \cite[(2.5) and
Corollary 2.1 (ii)]{Chen}.

\begin{lemma}\label{L:solu-2} For every $t,s>0$,
$$\int_0^t w(t-r)\bP (S_s>r)\,dr=G(t)-\bE (G(t-S_s){\bf 1}_{\{t\ge S_s\}})$$  and
$$\int_0^\infty \bE (G(t-S_r){\bf 1}_{\{S_r\le t\}})\,dr=t.$$
\end{lemma}

\medskip

 Now we can present the main result of this section on
 the existence and the uniqueness of  weak solutions to equation \eqref{e:1.3}.

\medskip

\begin{theorem}\label{T:2.4}
Suppose $M$ is a locally compact Hausdorff space and
 $X$ is a $\mu$-symmetric  strong Markov process for some $\sigma$-finite measure $\mu$ with full support on $M$.
  For any $f\in L^2(M; \mu)$,   $u(t,x):=\bE  \left[ T_{E_t} f(x)\right] $ is a weak solution to
\begin{equation}\label{e:2.1}
\partial_t^w u(t,x)=\sL u(t,x) \quad \hbox{with } u(0,x)=f(x)
\end{equation}
in the following sense:
\begin{itemize}
\item[\rm (i)]  $t\mapsto u(t, x)$ is continuous in $L^2(M; \mu)$.  Consequently, for every $t>0$,
$$
I^w_t (u (\cdot, x))  :=\int_0^t w(t-s)
(u(s, x) -f(x))\, ds
$$
 is absolutely convergent in   $L^2(M; \mu)$.

\item[\rm (ii)] For  every $g\in D(\sL)$ and $t>0$,
\begin{equation}\label{e:2.2}
\frac{d}{dt} \int_M g(x) I^w_t(u (\cdot, x))\,\mu(dx)= \int_M u(t,x) \sL g(x) \,\mu(dx)   .
\end{equation}
\end{itemize}

 Conversely, if $u(t, x)$ is a weak solution to \eqref{e:2.1} in the sense of  {\rm (i)}  and {\rm (ii)}  above with $f\in  L^2(M; \mu)$,
 then $u(t, x)= \bE \left[ T_{E_t} f (x) \right]$ $\mu$-a.e.\ on $M$
for every $t\geq 0$.

\end{theorem}

\begin{proof} The proof is motivated by that of
\cite[Theorem 2.1]{Chen}.

\medskip

(1) {\bf(Existence)} Since $\{T_t: t\geq 0\}$ is a strongly continuous contraction semigroup in $L^2(M;\mu)$ and
$t\mapsto  E_t$ is continuous a.s., we have by the bounded convergence theorem that
$t\mapsto u(t, x)=\bE \left[ T_{E_t} f (x) \right]$ is continuous in $L^2(M; \mu)$
and $\| u(t, \cdot)\|_2 \leq \| f\|_2$.
Since
\begin{equation}\label{e:2.3a}
\int_0^t w(s)\,ds  =   \int_0^\infty (z\wedge t)\, \nu (dz) <\infty \quad \hbox{for every } t>0
\end{equation}
by \cite[(2.2)]{Chen},
$I^w_t ( u(\cdot, x))$ is absolutely convergent in   $L^2(M; \mu)$ for every $t>0$ with
$\| I^w_t ( u(\cdot, x)) \|_2 \leq 2\| f\|_2 \int_0^t w(s) \,ds < \infty$.

In the following, denote by $\langle \cdot, \cdot \rangle$ the inner product in $L^2(M;\mu)$.
By \eqref{e:sou}, the integration by parts formula, Lemma \ref{L:solu-2} and the self-adjointness of  $\sL$ in $L^2(M; \mu)$, we have for every $t>0$,
\begin{align*}
&\int_M g(x) I^w_t (u (\cdot, x))\,\mu(dx)\\
&=\int_M g(x) \int_0^t w(t-r) (u(r, x) -u(0, x))\, dr\,\mu(dx)\\
&=\int_0^t w(t-r)\int_0^\infty \left(\langle T_s f,g\rangle-\langle f, g\rangle\right) \,d_s\bP (S_s\ge r)\,dr\\
&=\int_0^\infty \left(\langle T_s f,g\rangle-\langle f, g\rangle\right)\,d_s\left(\int_0^t w(t-r)\bP (S_s\ge r)\,dr\right)\\
&=-\int_0^\infty \left(\langle T_s f,g\rangle-\langle f, g\rangle\right)\,d_s \bE (G(t-S_s){\bf 1}_{\{S_s\le t\}})\\
&=\int_0^\infty \bE (G(t-S_s){\bf 1}_{\{S_s\le t\}}) \langle \sL T_sf, g\rangle\,ds\\
&=\int_0^\infty \bE (G(t-S_s){\bf 1}_{\{S_s\le t\}}) \langle T_sf, \sL g\rangle\,ds.
\end{align*}

On the other hand, according to \eqref{e:sou}, the integration by parts formula and Lemma \ref{L:solu-1}, we find that  for almost all $t>0$,
\begin{align*}
\int_0^t\int_M u(s,x) \sL g(x) \,\mu(dx)\,ds
&= \int_0^t \big\langle \sL g, \int_0^\infty T_u f\, d_u\bP (S_u\ge s)\big\rangle \,ds\\
&=\int_0^\infty \langle \sL g, T_u f\rangle \, d_u\left(\int_0^t \bP (S_u\ge s)\,ds\right)\\
&= \int_0^\infty \langle \sL g, T_u f\rangle \int_0^t \bE (w(s-S_u){\bf 1}_{\{S_u\le s\}})\,ds \, du\\
&= \int_0^\infty \langle \sL g, T_u f\rangle \bE (G(t-S_u){\bf 1}_{\{S_u\le t\}})\,du .
\end{align*}
Thus we conclude that for every $t\geq 0$,
$$
\int_M g(x) I^w_t (u (\cdot, x))\,\mu(dx) = \int_0^t\int_M u(s,x) \sL g(x) \,\mu(dx)\,ds .
$$
This establishes  \eqref{e:2.2} as $s\mapsto u(s, x)$ is continuous in $L^2(M; \mu)$.

\medskip

(2) {\bf(Uniqueness)} Suppose that $u(t, x)$ is a weak solution to \eqref{e:2.1} in the sense of (i) and (ii) with $f\in L^2(M; \mu)$.
Then $v(t, x):= u(t, x)- \bE \left[ T_{E_t} f (x) \right]$ is a weak solution to \eqref{e:2.1} with $v(0, x)=0$.
Note that by \eqref{e:2.3a},
$$  \lim_{t\to 0} \| I^w_t (v(\cdot, x)) \|_2
\leq 2\max_{s\in [0, 1]}  \| v(s,\cdot)\|_2\cdot   \lim_{t\to 0} \int_0^t w(s)\, ds  =0.
$$
Hence we have for every $t>0$ and $g\in \sD (\sL)$,
\begin{equation}\label{e:2.3}
 \int_M g(x) \left( \int_0^t w(t-r)  v(r, x) \,dr \right) \mu (dx)  = \int_M \left( \int_0^t v(s, x)  \,ds \right)\sL g( x)   \,\mu (dx) .
\end{equation}
Let $V(\lambda, x):=\int_0^\infty e^{-\lambda t} v(t, x)\, dt$, $\lambda >0$,
 be the Laplace transform of $t\mapsto v(t, x)$.
 Taking the Laplace transform in $t$ on both sides of \eqref{e:2.3} yields that for every $\lambda >0$,
$$
\int_M g(x) V(\lambda, x) \left(   \int_0^\infty e^{-\lambda s} w(s)\, ds \right) \,\mu(dx)
= \frac{1}{\lambda} \int_M V(\lambda, x) \sL g(x) \,\mu (dx).
$$
Note that the Laplace transform of $w(t)$ is $\phi (\lambda)/\lambda$; see \cite[(2.3)]{Chen}.
Hence we have from the above display that for every $\lambda >0$,
 $$
 \int_M V(\lambda, x)  \left( \phi (\lambda ) -\sL \right) g(x) \,\mu (dx) =0.
 $$
 Denote by $\{G_\alpha: \alpha>0\}$ be the resolvent of the regular Dirichlet for $(\sE, \sF)$.
 For each fixed $\lambda >0$ and $h\in L^2(M; \mu)$, take $g:=G_{\phi (\lambda)} h$, which is
 in $\sD (\sL)$. Since $ ( \phi (\lambda ) -\sL ) g  =h$, we deduce
 that $\int_M V(\lambda , x ) h(x) \,\mu (dx)=0$ for every $h\in L^2(M; \mu)$.
 Therefore $V(\lambda, x)=0$ $\mu$-a.e. for every $\lambda >0$. By the uniqueness of the Laplace transform
 and the fact that $t\mapsto v(t, x)$ is continuous in $L^2(M; \mu)$,
 it follows that $v(t, x)=0$ a.e. for every $t>0$. In other words, $u(t, x)= \bE \left[ T_{E_t} f(x) \right]$
 $\mu$-a.e. on $M$ for every $t>0$.
\end{proof}

\section{Preliminary estimates}
\label{S:3}

In this section, we give some preliminary estimates needed for the proofs of  Theorems \ref{theorem:mainjump} and \ref{theorem:maindiff}.

\subsection{Bernstein functions and the weak scaling property}\label{section31}

A non-negative $C^\infty$ function $\phi$ on $(0,\infty)$ is called a Bernstein function if
$(-1)^n \phi^{(n)} (\lambda) \le 0$
for every $n\in \mathbb N$ and $\lambda>0$.
According to \cite[(2.3)]{Ante} and \cite[Lemma 1.3]{KM},
the following properties hold for
the Bernstein function $\phi$
satisfying condition  \eqref{e:phi}.

\begin{lemma}
\label{l:AKM} Let $\phi$ be a Bernstein function such that \eqref{e:phi} is satisfied, i.e., there are constants $0<\beta_1\le \beta_2<1$ such that for any $\lambda>0$ and $\kappa\ge 1$,
$$ c_1 \kappa^{\beta_1}\le \frac{\phi(\kappa \lambda)}{\phi(\lambda)}\le c_2 \kappa^{\beta_2}.$$
Then there exists a constant $C_*\ge 1$ such that
the following holds
\begin{equation}\label{e:04}
 \lambda \, \phi'(\lambda) \le \phi(\lambda) \le C_* \lambda \, \phi'(\lambda),
\quad \lambda>0.
\end{equation}
In particular, there exist constants $c_i>0$ $(i=3,4,5,6)$ such that
\begin{equation}\label{e:031}c_3 \kappa^{1-\beta_2}\le \frac{\phi'(\lambda)}{\phi'(\kappa \lambda)}\le c_4 \kappa^{1-\beta_1},\quad \lambda>0, \kappa\ge 1,\end{equation}
and
\begin{equation}\label{e:03}c_5 \kappa^{1/(1-\beta_1)}\le \frac{(\phi')^{-1}(\lambda)}{(\phi')^{-1}(\kappa \lambda)}\le c_6 \kappa^{1/(1-\beta_2)},\quad \lambda>0, \kappa\ge 1,\end{equation} where $(\phi')^{-1}(\lambda):=\inf\{s>0: \phi'(s)\le \lambda\}$ for all $\lambda\ge0.$
\end{lemma}

A function $f :(0,\infty) \to \R$  is said to be  a completely monotone function if $f$ is smooth and
$(-1)^n f^{(n)} (\lambda)\ge 0$
for all
$n\in \mathbb N$ and $\lambda>0$. A Bernstein function is said to be a complete Bernstein function if its
L\'evy measure
has a completely monotone density  with respect to Lebesgue measure.
The next lemma is concerned with the weak scaling property, which is interesting of its own.

\begin{lemma}\label{p:dd1}
Suppose that $0<\alpha_1 \le \alpha_2 < \infty$ and that
 a family  of non-negative functions $\{\Phi(x, \cdot)\}_{x\in M}$
 satisfies the weak scaling property uniformly with $(\alpha_1,\alpha_2)$, i.e., there exist constants $c_1,c_2>0$   such that for any $x\in M$,
 \begin{equation}\label{e:llee}
c_1(R/r)^{\alpha_1}
\le\Phi(x,R)/\Phi(x,r) \le c_2(R/r)^{\alpha_2} , \quad 0<r \le R < \infty.
\end{equation}
Then for any $\alpha_3 >\alpha_2$, there is a
family  of complete Bernstein functions $\{\varphi(x, \cdot)\}_{x\in M}$
such that
$$\Phi(x,r) \simeq \frac{1}{\varphi(x,r^{-\alpha_3})},\quad r>0, x\in M.
$$
Consequently,
 $\{\varphi(x, \cdot)\}_{x\in M}$
enjoys the weak scaling property uniformly with
$(\alpha_1/\alpha_3,$ $\alpha_2/\alpha_3)$,
i.e., there are constants $c_3, c_4>0$ such that for all $x\in M$,
 \begin{align}
 \label{e:dd1}
c_3(R/r)^{\alpha_1/\alpha_3}
\le  \varphi(x,R)/ \varphi(x,r) \le c_4(R/r)^{\alpha_2/\alpha_3}  , \quad 0<r \le R < \infty
 \end{align}
\end{lemma}

\begin{proof} For any fixed $\alpha_3>\alpha_2$ and $x\in M$, define
$$
\varphi(x,\lambda)=\int^{\infty}_{0} \frac{\lambda}{\lambda +s}  \, \frac1{s \Phi(x,s^{-1/\alpha_3}) }\,
ds,\quad \lambda\geq 0,
$$
and
$$
\hat\Phi(x,u)=\int_0^\infty\frac{ e^{-us}}{\Phi(x,s^{-1/\alpha_3})}\,ds,\quad u\ge0.$$
Since
$$
\int^{\infty}_{0}(1-e^{-\lambda u}) e^{-us}  \,du=
\int^{\infty}_{0} e^{-su}\, du-\int^{\infty}_{0}e^{-(\lambda+s) u}  \, du=
\frac{\lambda}{s(\lambda+s)},
$$
we have
$$
\varphi(x,\lambda) = \int^{\infty}_{0}
 \int_0^\infty(1-e^{-\lambda u}) e^{-us}\,du \frac{1}{\Phi(x,s^{-1/\alpha_3})}
\,ds= \int^{\infty}_{0}   (1-e^{-\lambda u})\hat\Phi(x,u)\,du.
$$
In particular, $\hat\Phi(x,\cdot)$ is a completely monotone function, and so
$\varphi(x,\cdot)$ is a complete Bernstein function.

By the change of  variable $u=s^{-1/\alpha_3}$, we have that for any $x\in M$ and $\lambda>0$,
\begin{align*}
\varphi(x,\lambda)=\alpha_3\int^{\infty}_{0} \frac{\lambda u^{\alpha_3}}{\lambda u^{\alpha_3}+1}
\frac1{u\Phi(x,u)}
\,du
\simeq
 \psi(x,\lambda), \end{align*}
where
\begin{align*}
\psi(x,\lambda):  = \int^{\infty}_{0} (1\wedge (\lambda u^{\alpha_3}))
 \frac{1}{u\Phi(x,u)}\,du.
\end{align*}

Note that for all $x\in M$ and $\lambda>0$,
\begin{align*}
\psi(x,\lambda)\Phi(x,\lambda^{-1/\alpha_3})&= \int^{\infty}_{0} (1\wedge (\lambda u^{\alpha_3}))
 \frac{\Phi(x,\lambda^{-1/\alpha_3})}{u\Phi(x,u)}\,du\\
 & =\lambda\int_0^{\lambda^{-1/\alpha_3}} u^{\alpha_3-1}
 \frac{\Phi(x,\lambda^{-1/\alpha_3})}{\Phi(x,u)}\,du+\int_{\lambda^{-1/\alpha_3}}^\infty
 \frac{\Phi(x,\lambda^{-1/\alpha_3})}{u\Phi(x,u)}\,du.
\end{align*}
Using
\eqref{e:llee},
we can find that for all $x\in M$ and $\lambda>0$,
\begin{align*}
\frac{c_1}{\alpha_3-\alpha_1} \lambda^{-1}&=c_1  \lambda^{-\alpha_1/\alpha_3} \int_0^{\lambda^{-1/\alpha_3}} u^{\alpha_3-1-\alpha_1}\,
 du   \\
 &\le \int_0^{\lambda^{-1/\alpha_3}} u^{\alpha_3-1}
 \frac{\Phi(x,\lambda^{-1/\alpha_3})}{\Phi(x,u)}\,du \\
 &\le c_2  \lambda^{-\alpha_2/\alpha_3} \int_0^{\lambda^{-1/\alpha_3}}  u^{\alpha_3-1-\alpha_2}\,
 du=\frac{c_2}{\alpha_3-\alpha_2} \lambda^{-1}
 \end{align*}
and  \begin{align*}
\frac{1}{c_2\alpha_2}&=c_2^{-1} \lambda^{-\alpha_2/\alpha_3} \int_{\lambda^{-1/\alpha_3}}^\infty u^{-1-\alpha_2}\,
 du  \\
 &\le \int_{\lambda^{-1/\alpha_3}}^\infty
 \frac{\Phi(x,\lambda^{-1/\alpha_3})}{ u\Phi(x,u)}\,du \le c_1^{-1} \lambda^{-\alpha_1/\alpha_3} \int_{\lambda^{-1/\alpha_3}}^\infty u^{-1-\alpha_1}\, du=\frac{1}{c_1\alpha_1}.
 \end{align*}
 Therefore, for all $x\in M$ and $\lambda>0$,
$$
 \varphi(x,\lambda) \simeq \frac{1}{\Phi(x,\lambda^{-1/\alpha_3})},
$$
 which along with \eqref{e:llee} yields
\eqref{e:dd1}. The proof is complete.\end{proof}

By Lemma \ref{p:dd1} above, for any function $\Phi(x,r)$ satisfying \eqref{e:llee},
we have
$$
\Phi(x,r)\simeq \tilde\Phi(x,r):= 1/ \varphi(x,r^{-\alpha_3})
$$
for some complete Bernstein function $\varphi(x,\cdot)$ and  $\alpha_3>\alpha_2$.
According to \eqref{e:04}, for all $x\in M$ and $r>0$,
$$
r \partial_r \tilde \Phi(x,r)
=\alpha_3 \frac{\partial_r \varphi(x,r^{-\alpha_3})}{\varphi(x,r^{-\alpha_3})^2} \, r^{-\alpha_3} \simeq  \frac{1}{\varphi(x,r^{-\alpha_3})} =\tilde\Phi(x,r)
$$
and so, by the inverse function theorem with $t=\tilde\Phi(x, s)$, for all $x\in M$ and $t>0$,
\begin{align}
\label{e:Phiinvesd}
\frac{(\partial_t \tilde\Phi^{-1}(x,\cdot))(t)}{\tilde\Phi^{-1}(x,\cdot)( t)} =\frac{(\partial_t\tilde\Phi^{-1}(x,\cdot))(\tilde\Phi(x, s))}{s}=
\frac{1}{s \partial_s \tilde \Phi(x, s)}
\simeq\frac{1}{\tilde\Phi(x, s)}     =  \frac1t.
\end{align}

\subsection{{\bf Estimates for subordinator}}

\begin{proposition}\label{p:sub}  Let $\{S_t:t\ge0\}$ be a subordinator whose Laplace exponent
 $\phi$ satisfies assumption \eqref{e:phi}.
\begin{itemize}
\item[(i)] There are constants $c_1,c_2>0$ such that for all $r,t\ge0$,
\begin{equation}\label{s:es-1}
\bP  \left( S_r\ge t(1+er\phi(t^{-1}) \right) \le c_1 r\phi(t^{-1})
\end{equation}
and
\begin{equation}\label{s:es-2}
\bP (S_r\geq  t)
\ge 1- e^{-c_2r \phi(t^{-1})}.\end{equation}
 In particular, for each $L>0$, there exist constants $c_{1,L},c_{2,L} >0$ such that
for all $r\phi(t^{-1})\le L$,
$$
c_{1,L} r\phi(t^{-1})\le \bP (S_r\ge t)\le  c_{2,L} r\phi(t^{-1}).
$$

\item[(ii)] There is a constant $c_1>0$ such that for all $r, t>0$,
$$ \bP (S_r\le t)\le  \exp(-c_1 r \phi\circ[(\phi')^{-1}] (t/r)) \le  \exp(-c_1t (\phi')^{-1}(t/r)).$$
Moreover, there is a constant $c_0>0$ such that for each $L>0$, there exists a constant $c_{c_0,L}>0$ so that  for  $r\phi(t^{-1}) > L$
 $$ \bP (S_r\le t)\ge c_{c_0,L} \exp\left(-c_0 r \phi\circ[(\phi')^{-1}] (t/r)\right)\ge  c_{c_0,L} \exp(- c_0 C_*t (\phi')^{-1}(t/r)),$$ where $C_*>0$ is the constant in \eqref{e:04}.\end{itemize}

\end{proposition}

\begin{proof} (i) \eqref{s:es-1} and \eqref{s:es-2} follow from \cite[Propositions 2.3 and 2.9]{Ante} and \cite[Proposition 2.5, Lemma 2.6 and Proposition 2.9]{Ante}, respectively. The last assertion is a direct consequence of
\eqref{e:phi}, \eqref{s:es-1} and \eqref{s:es-2}.

(ii) According to \eqref{e:031}, we have $\phi'(0)=\infty$ and so $$\int_0^\infty s\,\nu(ds)=\phi'(0)=\infty.$$
Since
$$r \cdot \phi'\circ[(\phi')^{-1}](t/r) \cdot (\phi')^{-1}(t/r)= t \cdot (\phi')^{-1}(t/r),$$
by \eqref{e:04}
\begin{align}
\label{e:pac}
 t \cdot (\phi')^{-1}(t/r) \le r \phi\circ[(\phi')^{-1}] (t/r) \le
 C_* t \cdot (\phi')^{-1}(t/r).
\end{align}
Now, the first assertion follows from \eqref{e:pac},
\cite[Lemma 5.2]{JP}
and \cite[Proposition 2.9]{Ante}.

On the other hand, by \cite[Lemma 5.2]{JP} and \cite[Proposition 2.9]{Ante} again,
there exist constants $c_0, c_1, c_2>0$ $(c_0$ is independent of $c_1$ and $c_2$) such that
for $r \phi\circ[(\phi')^{-1}] (t/r) \ge  c_1$,
\begin{equation*}
\label{e:pac1} \bP (S_r\le t)\ge c_2 \exp\Big(-c_0 r \phi\circ[(\phi')^{-1}](t/r)\Big). \end{equation*}
Thus, according to \eqref{e:03}, \eqref{e:pac} and \eqref{e:04}, we see that there exists a constant $c_3>0$ such that for  $r\phi(t^{-1})  \ge c_3$ (so that $\phi'(t^{-1}) \ge (c_3/C_*)t/r$)
\begin{equation} \label{e:pac2}
\bP (S_r\le t)\ge c_2 \exp\Big(-c_0 r \phi\circ[(\phi')^{-1}](t/r)\Big) \ge  c_2 \exp\Big(-
c_0C_*t (\phi')^{-1}(t/r)\Big).
\end{equation}
We observe that,  if  $L<r\phi(t^{-1})  \le c_3$ for a constant $L>0$, then
by \eqref{e:phi} and \eqref{e:04}
$$ r \phi\circ[(\phi')^{-1}](t/r)\le  \frac{c_3\phi\circ[(\phi')^{-1}](t\phi(t^{-1})/c_3)}{\phi(t^{-1})}=  \frac{c_3\phi\circ[(\phi')^{-1}]\left(\frac{\phi(t^{-1})}{c_3t^{-1}}\right)}{\phi\circ[(\phi')^{-1}] (\phi'(t^{-1}))}\simeq 1, $$
and
$$ r \phi\circ[(\phi')^{-1}](t/r)\ge  \frac{L\phi\circ[(\phi')^{-1}](t\phi(t^{-1})/L)}{\phi(t^{-1})}=  \frac{L\phi\circ[(\phi')^{-1}]\left(\frac{\phi(t^{-1})}{Lt^{-1}}\right)}{\phi\circ[(\phi')^{-1}] (\phi'(t^{-1}))}\simeq 1, $$
Thus, using \eqref{e:pac2} for $r\phi(t^{-1})  = c_3$, we have that for any $L>0$ such that
 $L<r\phi(t^{-1})  \le c_3$,
 $$\bP (S_r\le t)\ge  \bP (S_{c_3/\phi(t^{-1})}\le t) \ge c_2   e^{-c_4} \ge c_2     \exp\big(-c_5 r \phi\circ[(\phi')^{-1}](t/r)\big). $$
This completes the proof.
\end{proof}

\subsection{{\bf
Preliminary lower bound estimates
for $p(t,x,y)$}}
Recall that
$X$ is a strong Markov process on the locally compact separable metric measure space $( M,d, \mu ) $ having infinite lifetime and
$q(t, x, y)$  is the transition density function of $X$ with respect to $\mu$.

The next statement is a key lemma used in the proof of the lower bound for $p(t,x,y)$, which is defined in \eqref{e:1.5}.

\begin{lemma} \label{L:lower1m}
Suppose that for each $T>0$ there exists a constant $c_0=c_0(T)>0$ such that
\begin{equation}\label{e:hklowi}
q(t,x,y)\ge  \frac{c_0}{V(x, \Phi^{-1}(t))}
\quad \text{for all }  x,y\in M \hbox{ and }
t \in(0, T\Phi(d(x,y))],
\end{equation}
where $\Phi :[0,+\infty )\rightarrow \lbrack 0,+\infty )$
is a strictly increasing function with $\Phi(0)=0$
and satisfies the weak scaling property with $(\alpha_1,\alpha_2)$
for some constants $0<\alpha_1\le \alpha_2<\infty$.
Then for every $L>0$,
there is a constant $c_1:=c_1(L)>0$ such that for all $x,y\in M$ and $t>0$ with $\Phi( d(x, y))\phi(t^{-1})\le L$,
\begin{align*}
p(t,x, y)&\ge c_1 \left(\frac{1}{V(x,\Phi^{-1}(1/\phi(t^{-1})))} \vee   \frac{\Phi(d(x, y)) \phi(t^{-1})}{ V(x,d(x, y)) }  \right).
\end{align*}
\end{lemma}

\begin{proof}
 By \eqref{s:es-1} and \eqref{s:es-2} in Proposition \ref{p:sub} (i),
 we can choose constants $\kappa_1>\kappa_2>0$ such that for all $t>0$,
\begin{equation}\label{eeff11}\bP (S_{\kappa_1/\phi(t^{-1})}\ge t)- \bP (S_{\kappa_2/\phi(t^{-1})}\ge t)\ge c_0.\end{equation}
Inequality \eqref{eeff11}
along with \eqref{e:1.5} yields that for every $L>0$, $x,y\in M$ and $t>0$ with $\Phi( d(x, y))\phi(t^{-1})\le L$
\begin{equation}\label{e:lowpre1} \begin{split}
&p(t, x, y)\ge c\int_{\kappa_2/\phi(t^{-1})}^{\kappa_1/\phi(t^{-1})}  q(r,x, y)\,d_r\bP (S_r\ge t)\\
&\ge c \left(\min_{\kappa_2/\phi(t^{-1})\le r\le \kappa_1/\phi(t^{-1})}  q(r,x, y)\right) \left(\bP (S_{\kappa_1/\phi(t^{-1})}\ge t)- \bP (S_{\kappa_2/\phi(t^{-1})}\ge t)\right)\\
& \ge\frac{c}{V(x,\Phi^{-1}(1/\phi(t^{-1})))},
\end{split}\end{equation}
where in the last inequality we have used \eqref{e:hklowi} and the fact that
$\Phi(d(x,y))\phi(t^{-1})\le L$.
Similarly,
according to \eqref{s:es-1} and \eqref{s:es-2} in Proposition \ref{p:sub} (i),
one can choose constants $\kappa_3>\kappa_4>0$ such that for all $t>0$ and $z\ge0$,
\begin{align}
\label{e:Ssecond}
\bP (S_{\kappa_3\Phi(z)}\ge t)- \bP (S_{\kappa_4 \Phi(z)}\ge t)\ge c_1 \Phi(z)\phi(t^{-1}).
\end{align}
Using \eqref{e:Ssecond} and
the argument of \eqref{e:lowpre1}, we find that for every $L>0$, $x,y\in M$ and $t>0$ with $\Phi( d(x, y))\phi(t^{-1})\le L$
\begin{align*}
&p(t,x,y)\ge c\int_{\kappa_4\Phi(d(x,y))}^{\kappa_3\Phi(d(x,y))} q(r,x, y)\,d_r\bP (S_r\ge t)\\
&\ge c\left( \min_{\kappa_4\Phi(d(x,y))\le r\le \kappa_3 \Phi(d(x,y))} q(r,x, y) \right)\left(\bP (S_{\kappa_3\Phi(d(x,y))}\ge t)- \bP (S_{\kappa_4\Phi(d(x,y))}\ge t)\right)\\
&\ge \frac{c\Phi(d(x, y)) \phi(t^{-1})}{ V(x,d(x, y))}.
\end{align*}
The proof is complete.
\end{proof}

\section{Non-local spatial motions}\label{S:4}

\subsection{Time derivative of heat kernel estimates for jump process}

In this section, we consider the pure jump case where $p(t,x,y)$ satisfies
\eqref{e:hkmequi}. First, note that since $\Phi$
 is strictly increasing and
 satisfies the weak scaling property with $(\alpha_1,\alpha_2)$, there are constants $c_1,c_2>0$ such that for all $\kappa\ge 1$ and $\lambda>0$,
\begin{equation}\label{e:Phii} c_1 \kappa^{1/\alpha_2}\le \frac{\Phi^{-1}(\kappa \lambda)}{\Phi^{-1}(\lambda)}\le c_2 \kappa^{1/\alpha_1}.\end{equation}

Set
\begin{equation}\label{e:4.6}
\bar q(t, x, r) :=\frac{t}{t V(x,\Phi^{-1}(t)) + \Phi (r) \, V(x,r)}, \quad t, r >0 \hbox{ and } x\in M.
\end{equation}
Note that by \eqref{e:hkmequi} and the fact $1 \wedge (1/r)  \simeq 1/(1+r) $ for $r>0$, we have
\begin{equation}\label{e:4.7}
q(t, x, y) \simeq \bar q(t, x, d(x, y)) \quad \hbox{for every } t>0 \hbox{ and } x, y \in M.
\end{equation}
According to Lemma \ref{p:dd1} and the remark at the end of Subsection \ref{section31},
we may and do
assume that both $V(x,\cdot)$ and $\Phi(\cdot)$ are differentiable and satisfy the property like \eqref{e:Phiinvesd}.

We   next give a lemma concerning the time derivative of
$\bar q (t, x, r)$.
\begin{lemma}\label{L:gd} Under assumptions above,
there is a constant $c_1>0$ such that for all $t,r>0$ and $x\in M$,
\begin{equation}\label{e:02m}
\left|\frac{\partial \bar q(t,x, r)}{\partial t}\right| \le c_1   \frac{ \bar q(t,x, r)}t,
\end{equation}
and that there exist constants $c_2, c_3>0$, $c_*\in (0,1)$ and  $c^*\in (1,\infty)$ such that for all $x\in M$,
\begin{align}\label{e:04m1}
\frac{\partial \bar q(t,x, r)}{\partial t} \le -c_2
 \frac{\bar q(t,x, r)}t \quad \text{if }\,\,\Phi(r) \le c_* t,
\end{align}
and
\begin{align}\label{e:04m2}
\frac{\partial \bar q(t,x, r)}{\partial t} \ge c_3  \frac{\bar  q(t,x, r)}t \quad \text{if }\,\,  \Phi(r) \ge c^* t.
\end{align}\end{lemma}
\begin{proof}
By elementary calculations, we have
\begin{align*} &\frac{\partial \bar q(t,x, r)}{\partial t}\\
  &=\frac{(t V(x,\Phi^{-1}(t)) + \Phi (r) \, V(x,r))-t[V(x,\Phi^{-1}(t))+t \partial_r V(x,\Phi^{-1}(t))(\Phi^{-1}(t))']}{(t V(x,\Phi^{-1}(t)) + \Phi (r) \, V(x,r))^2}\\
 &=\frac{ \Phi (r) \, V(x,r)-t^2 \partial_r V(x,\Phi^{-1}(t))(\Phi^{-1}(t))'}{(t V(x,\Phi^{-1}(t)) + \Phi (r) \, V(x,r))^2}\\
  &=\frac{\bar q(t,x, r)}{t}\left(\frac{ \Phi (r) \, V(x,r)-t^2 \partial_r V(x,\Phi^{-1}(t))(\Phi^{-1}(t))'}{t V(x,\Phi^{-1}(t)) + \Phi (r) \, V(x,r)} \right).
\end{align*}
Since
$t^2 \partial_r V(x,\Phi^{-1}(t))(\Phi^{-1}(t))'  \simeq tV(x,\Phi^{-1}(t))$ by \eqref{e:Phiinvesd},
we have
\begin{align*}
 \frac{\bar q(t,x, r)}{t}&\left( \frac{ \Phi (r) \, V(x,r)-c_1 tV(x,\Phi^{-1}(t))}{t V(x,\Phi^{-1}(t)) + \Phi (r) \, V(x,r)}\right) \\
 &\le
\frac{\partial \bar q(t,x, r)}{\partial t} \le  \frac{\bar q(t,x, r)}{t}\left( \frac{ \Phi (r) \, V(x,r)-c_2 tV(x,\Phi^{-1}(t))}{t V(x,\Phi^{-1}(t)) + \Phi (r) \, V(x,r)}\right). \end{align*}  Thus, the desired assertion follows from the estimate above.
   \end{proof}

\subsection{Two-sided estimates for $p(t,x,y)$}\label{s:4.2}

Recall that for $t>0$ and $x, y\in M$,
$$p(t,x,y)=\int_0^\infty q(r, x, y) \, d_r \bP (E_t \leq r) = \int_0^\infty  q(r,x,y) \,d_r \bP(S_r\geq t)  .
$$

\begin{proof}[Proof of Theorem $\ref{theorem:mainjump}$] Throughout the proof, we fix $x,y\in M$. By \eqref{e:4.7}, \begin{equation*}\label{eq:nroo3}
p(t,x,y)=\int_0^\infty q(r,x,y)\,d_r\bP (S_r\ge t)
\simeq
\int_0^\infty   \bar q(r, x, d(x, y))\,d_r\bP (S_r\ge t).
\end{equation*} Then,
for $t>0$ and $x, y\in M$,
\begin{equation}\label{e:07} \begin{split}
p(t,x, y )\simeq
&\int_0^{2/\phi(t^{-1})} \bar q(r,x, d(x, y)) \,d_r\bP (S_r\ge t) \\
&-\int_{2/\phi(t^{-1})}^\infty  \bar q(r,x, d(x,y)) \,d_r \bP (S_r\le  t) \\
=&:I_1+I_2.\end{split}\end{equation}

For simplicity, in the following we fix $x\in M$ and let $z=d(x, y)$.
 Then by definition, $\bar q(t, x, d(x,y))=\bar q(t, x, z)$. We also write $\bar q(t,x,z)$ and $V(x,r)$ as $\bar q(t,z)$ and $V(r)$, respectively.
The proof is divided into two parts.

\medskip

\noindent{\bf Proof of the upper bound of $p(t,x,y)$. }\,
For $I_1$, since $\bP (S_0\ge t)=0$ for $t>0$ and $\bar q(0,\cdot)=\delta_{\{0\}}$ (this is understood in the usual way and $\delta_{\{0\}}$ is the Dirac measure at
 the point 0), we have
by Proposition \ref{p:sub} (i) and \eqref{e:02m}
\begin{equation}\label{e:07-1} \begin{split}
I_1=&
\bar q(r, z)
\bP (S_r\ge t) \big|_0^{2/\phi(t^{-1})} - \int_0^{2/\phi(t^{-1})} \bP (S_r\ge t)\,d_r \bar q(r, z)\\
\le & c \bar q(2/\phi(t^{-1}),
z) - \int_0^{2/\phi(t^{-1})} \bP (S_r\ge t)\,d_r \bar q(r, z)\\
\le &c \bar q(2/\phi(t^{-1}),
z)+c\int_0^{2/\phi(t^{-1})} r\phi (t^{-1}) \cdot \frac{1}{r} \cdot \bar q(r, z)\,dr\\
=&\!\!:c \bar q(2/\phi(t^{-1}),z)+c I_{1,1}.
\end{split}
\end{equation}
For $I_2$, since
$\bar q(\infty, z)=0$,
\begin{equation}\label{e:07-2} \begin{split}I_2=&-\int_{2/\phi(t^{-1})}^\infty  \bar q(r, z) \, d_r \bP (S_r\le t)\\
=&- \bar q(r, z)\bP (S_r\le t)|_{2/\phi(t^{-1})}^\infty +\int_{2/\phi(t^{-1})}^\infty \bP (S_r\le t)\,d_r \bar q(r, z)\\
\le &c \bar q(2/\phi(t^{-1}),z)+ c\int_{2/\phi(t^{-1})}^\infty  \exp(-c_1
 t (\phi')^{-1}(t/r)) \cdot \frac{1}{r} \cdot \bar q(r,z)\,dr\\
=&\!\!:c \bar q(2/\phi(t^{-1}),z)+ c I_{2,1},
\end{split}\end{equation}
where in
the inequality above we used Proposition \ref{p:sub} (ii) and \eqref{e:02m}. Therefore, in order to get upper bound of $p(t,x,y)$,
we need to derive upper bound
for $I_{1,1}$ and $I_{2,1}$.

\noindent
{\bf(1-a)}  Suppose that $\Phi(z)\phi(t^{-1})\le 1$. Then
by \eqref{e:4.6}
\begin{align*}I_{1,1} &\le c
\phi(t^{-1}) \left(\frac{1}{ V(z) \Phi(z)} \int_0^{\Phi(z)}r\,dr+  \int_{\Phi(z)}^{2/\phi(t^{-1})} \frac 1{V(\Phi^{-1}(r))}\, dr \right)\\
&=  c\phi(t^{-1})\left(\frac{\Phi(z)}{V(z)} +\int_{\Phi(z)}^{2/\phi(t^{-1})} \frac 1{V(\Phi^{-1}(r))} \,dr\right)\\
&\le c \phi(t^{-1})\int_{\Phi(z)}^{2/\phi(t^{-1})}\frac 1{V(\Phi^{-1}(r))} \,dr,
\end{align*} where in the last inequality we used the fact that
\begin{equation}\label{e:rrffee1}\int_{\Phi(z)}^{2/\phi(t^{-1})}\frac 1{V(\Phi^{-1}(r))} \,dr\ge \int_{\Phi(z)}^{2\Phi(z)}\frac 1{V(\Phi^{-1}(r))} \,dr \ge c \,\frac{\Phi(z)}{V(z)}.\end{equation}

By changing the variable $s=r\phi(t^{-1})$ and using \eqref{vd1} and \eqref{e:Phii}, we find that
\begin{equation}\label{e:com1}\begin{split}
&\phi(t^{-1}) \int_{\Phi(z)}^{2/\phi(t^{-1})} \frac 1{V(\Phi^{-1}(r))} \,dr\\
&=\int_{ \Phi(z) \phi(t^{-1})}^{2 } \frac 1{V(\Phi^{-1}( s/\phi(t^{-1})))} \,ds\\
&=
\frac{1}{V(\Phi^{-1}(1/\phi(t^{-1})))}
\int_{ \Phi(z) \phi(t^{-1})}^{2 }
\frac{V(\Phi^{-1}(1/\phi(t^{-1})))}{V(\Phi^{-1}( s/\phi(t^{-1})))}
\,ds\\
&\ge
\frac{c}{V(\Phi^{-1}(1/\phi(t^{-1})))}
\int_{ 1}^{2 } s^{-d_2/\alpha_1} \,ds\\
&\ge
\frac{c}{V(\Phi^{-1}(1/\phi(t^{-1})))}\ge c \bar q(2/\phi(t^{-1}),z).
\end{split}
\end{equation}
Hence
\begin{align*}I_1\le& c \phi(t^{-1}) \int_{\Phi(z)}^{2/\phi(t^{-1})} \frac 1{V(\Phi^{-1}(r))}
\,dr.\end{align*}

\noindent
{\bf(1-b)} If $\Phi(z)\phi(t^{-1})\ge 1$, then
by \eqref{e:4.6}
\begin{align*}I_{1,1}&\le c\phi(t^{-1})
\int_0^{2/\phi({t^{-1}})} \bar q(r,z)\,dr\le\frac{c\phi({t^{-1}})}{V(z) \Phi(z)}
\int_0^{2/\phi({t^{-1}})} r\,dr\le \frac{c}{ \phi({t^{-1}}) V(z) \Phi(z)}.
\end{align*}
Since
by \eqref{e:4.6} again
$$ \bar q(2/\phi(t^{-1}),z) \leq  \frac{c}{ V(z) \Phi(z)\phi({t^{-1}})},$$
we obtain $$I_1 \le  \frac{c}{\phi({t^{-1}}) V(z) \Phi(z)}.$$

\noindent
{\bf(2-a)} If  $\Phi(z)\phi(t^{-1})\le 1$, then
by changing variable $s=r \phi(t^{-1})$,
and using
 \eqref{e:4.6}, \eqref{vd1}, \eqref{e:Phii},  \eqref{e:04} and \eqref{e:03},
\begin{equation}\label{com1-100}\begin{split}
I_{2,1}&\le c\int_{2/\phi(t^{-1})}^\infty \exp(-c_1 t (\phi')^{-1}(t/r)) \cdot r^{-1}\cdot
{\bar q}(r,z)\,dr\\
&\le c  \int_{2/\phi(t^{-1})}^\infty \exp(-c_1 t (\phi')^{-1}(t/r)) \cdot r^{-1} \cdot \frac 1{V(  \Phi^{-1}(r))} \,dr\\
&=c  \int_{2}^\infty \exp(-c_1 t (\phi')^{-1}(t\phi(t^{-1})/s))  \cdot s^{-1}  \cdot \frac 1{V(\Phi^{-1}(s/ \phi(t^{-1})))} \,ds\\
& \le \frac{c}{V(\Phi^{-1}(1/\phi(t^{-1})))}\\
&\quad \times \int_2^\infty \exp(-c_1 t(\phi')^{-1}(\phi'(t^{-1})/s))\cdot
\frac{V(\Phi^{-1}(1/\phi(t^{-1})))} {V(\Phi^{-1}(s/ \phi(t^{-1})) ) }\cdot s^{-1}\, ds\\
& \le \frac{c}{V(\Phi^{-1}(1/\phi(t^{-1})))}\int_2^\infty \exp(-c_1 t(\phi')^{-1}(\phi'(t^{-1})/s))\,s^{-((d_1/\alpha_2)+1)}\,ds \\
&= \frac{c}{V(\Phi^{-1}(1/\phi(t^{-1})))}\\
&\quad \times\sum_{n=1}^\infty\int_{2^n}^{2^{n+1}} \exp(-c_1 t(\phi')^{-1}(\phi'(t^{-1})/s))\,s^{-((d_1/\alpha_2)+1)}\,ds\\
&\le\frac{c}{V(\Phi^{-1}(1/\phi(t^{-1})))}\sum_{n=1}^\infty \exp(-c_1 t(\phi')^{-1}(\phi'(t^{-1})/2^n))\,2^{-n((d_1/\alpha_2)+1)}\\
&\le \frac{c}{V(\Phi^{-1}(1/\phi(t^{-1})))}\sum_{n=1}^\infty \exp(-c_2 2^{n(1-\beta_2)})\,2^{-n((d_1/\alpha_2)+1)}\\
&\le \frac{c}{V(\Phi^{-1}(1/\phi(t^{-1})))}. \end{split}\end{equation}
Thus
 $$I_2\le  \frac{c}{V(\Phi^{-1}(1/\phi(t^{-1})))}.$$

\noindent
{\bf(2-b)}  Next, we suppose that  $\Phi(z)\phi(t^{-1})\ge 1$.
Following the same argument as
\eqref{com1-100}, we find that
\begin{align*}I_{2,1}&\le c  \int_{2/\phi(t^{-1})}^\infty \exp(-c_1 t (\phi')^{-1}(t/r)) \cdot r^{-1}\cdot {\bar q}(r,z)\,dr\\
&\le   \frac{c}{\Phi(z)V(z)} \int_{2/\phi(t^{-1})}^\infty \exp(-c_1 t (\phi')^{-1}(t/r))\,
dr\\
&\le  \frac{c}{\phi(t^{-1}) \Phi(z)V(z)}   \int_2^\infty \exp(-c_1 t(\phi')^{-1}(\phi'(t^{-1})/s))\,ds\le \frac{c}{\phi(t^{-1}) \Phi(z)V(z)}. \end{align*}
Thus,
$$I_2 \le \frac{c}{ \phi({t^{-1}})V(z) \Phi(z)}.$$

Combining all the estimates above, we have proved the desired upper bounded estimates for $p(t,x,y).$

\medskip

\noindent
{\bf Proof of the lower bound of $p(t,x,y)$.}\,\,
{\bf(1)} Assume that $\Phi(z)\phi(t^{-1})\le 1$. By \eqref{e:02m} and \eqref{e:04m1}, we can find a constant $c_1>1$ such that when $c_1 \Phi(z)\le r,$
$$ \frac{\partial {\bar q}(r,z) }{\partial r} \le -\frac{c_2}{r} {\bar q}(r,z) \le -\frac{c_3}{ rV(\Phi^{-1}(r))};$$ and when  $ 0< r\le c_1 \Phi(z)$,
$$  \frac{\partial {\bar q}(r,z) }{\partial r} \le  \frac{c_4}{r}{\bar q}(r,z).$$
Then since
 $$\int_{r_0}^\infty  {\bar q}(r,z) \,d_r \bP (S_r\le  t)   \le 0,\quad r_0>0,$$  according to the arguments of \eqref{e:07} and \eqref{e:07-1}, we have
\begin{align*}
p(t, x, y )\ge & c_5
\int_0^{2c_1/\phi(t^{-1})} \bar q(r,z) \,d_r\bP (S_r\ge t)\\
\ge &
-c_5 \int_{c_1 \Phi(z)}^{2c_1/\phi(t^{-1})} \bP (S_r\ge t)\,d_r \bar q(r, z)
-c_5 \int^{c_1 \Phi(z)}_0\bP (S_r\ge t)\,d_r \bar q(r, z)\\
\ge & c_6\phi(t^{-1}) \int_{c_1 \Phi(z)}^{2c_1/\phi(t^{-1})} \frac 1{V(\Phi^{-1}(r))}\,dr-\frac{c_7 \phi(t^{-1})}{V(z)\Phi(z)}\int^{c_1 \Phi(z)}_0     r\,dr\\
=&:I_{1,1}-I_{1,2}.\end{align*}
Noting that
\begin{align*}I_{1,2}
&\le  \frac{c_7c_1^2\phi(t^{-1})\Phi(z) }{2V(z)},
\end{align*}
and changing the variable $s=r\phi(t^{-1})$, we have
\begin{align*} p(t,x,y)
\ge  &c_6\int_{c_1 \Phi(z) \phi(t^{-1})}^{2c_1} \frac 1{V(\Phi^{-1}( s/\phi(t^{-1})))}\, ds - \frac{c_7c_1^2\phi(t^{-1})\Phi(z) }{2V(z)}.
 \end{align*}
Combining this with Lemma \ref{L:lower1m}  yields
\begin{align*} (c_*+1)p(t,x,y)
\ge  &
c_*c_6\int_{c_1 \Phi(z) \phi(t^{-1})}^{2c_1} \frac 1{V(\Phi^{-1}( s/\phi(t^{-1})))}\, ds\\
& - \frac{c_*c_7c_1^2\phi(t^{-1})\Phi(z) }{2V(z)}+\frac{c_8\phi(t^{-1})\Phi(z) }{V(z)}\\
\ge & c_*c_6\int_{c_1 \Phi(z) \phi(t^{-1})}^{2c_1} \frac 1{V(\Phi^{-1}( s/\phi(t^{-1})))}\, ds\\
\ge & c_9\int_{\Phi(z) \phi(t^{-1})}^{2} \frac 1{V(\Phi^{-1}( s/\phi(t^{-1})))}\, ds,
 \end{align*}
 where $c_*>0$ is chosen small enough so that $c_*c_7c_1^2<2c_8$.

\noindent
{\bf(2)} Next we assume that $\Phi(z)\phi(t^{-1})\ge 1$. Due to
\eqref{e:4.6},  \eqref{e:02m} and  \eqref{e:04m2}
we can find a constant $0<c_1<1$ such that when $ r \le c_1 \Phi(z),$
$$  \frac{\partial {\bar q}(r,z) }{\partial r} \ge  \frac{c_2}{r}{\bar q}(r,z);$$
 when $ r\ge c_1 \Phi(z)$,
$$
 \frac{\partial {\bar q}(r,z) }{\partial r} \ge  - \frac{c_3}{rV(\Phi^{-1}(r))}.
$$
Then according to \eqref{e:04} and the arguments of \eqref{e:07} and \eqref{e:07-2}, we know that
\begin{align*}p(t,x, y)
\ge & -c_4
\int_{c_1/(2\phi(t^{-1}))}^\infty \bar q(r, z) \,d_r\bP (S_r\le t)\\
\ge & c_4
\int^{c_1 \Phi(z)}_{c_1/(2\phi(t^{-1}))}
\bP (S_r\le t)\,d_r \bar q(r, z) +c_4
\int_{c_1 \Phi(z)}^\infty \bP (S_r\le t)\,d_r \bar q(r, z)
\\
\ge & \frac{ c_5}{V(z) \Phi(z)}\int^{c_1 \Phi(z)}_{c_1/(2\phi(t^{-1}))}  \exp(-c_6t (\phi')^{-1}(t/r))  \,dr\\
&-c_7\int_{c_1 \Phi(z)}^\infty \frac 1{rV(\Phi^{-1}(r))}\exp(-c_8t (\phi')^{-1}(t/r))\,dr\\
=&:I_{2,1}-I_{2,2}. \end{align*}

Changing the variable $s=r\phi(t^{-1})$ and using \eqref{e:04} and \eqref{e:03}, we have
\begin{equation}\label{e:I21}\begin{split}I_{2,1}&\ge   \frac {c_9}{\phi(t^{-1})V(z) \Phi(z)} \int_{c_1/2}^{c_1 \Phi(z) \phi(t^{-1})}  \exp(-c_6t(\phi')^{-1}(t\phi(t^{-1})/s))\,ds \\
&\ge \frac {c_9}{\phi(t^{-1})V(z) \Phi(z)} \int_{c_1/2}^{c_1 }  \exp(-c_6t(\phi')^{-1}(t\phi(t^{-1})/s))\,ds \\
&\ge \frac {c_9}{\phi(t^{-1})V(z) \Phi(z)} \exp(-c_6t(\phi')^{-1}(t\phi(t^{-1})/c_1))\ge  \frac {c_{10}} {\phi(t^{-1})V(z) \Phi(z)}
,\end{split} \end{equation}where we used
\eqref{e:04} and \eqref{e:03} in
the last  inequality above.
Here, we observe that by \eqref{vd1} and \eqref{e:Phii}
\begin{equation}\begin{split}\label{e:I21n}
\frac {1} {\phi(t^{-1})V(z) \Phi(z)}
 &= \frac 1{V(\Phi^{-1}(1/\phi(t^{-1})))}
\frac {1}{ \Phi(z)\phi(t^{-1})} \frac{V(\Phi^{-1}(1/\phi(t^{-1})))}{V(\Phi^{-1}(\Phi (z)))}\\
&\ge \frac{c_{11}}{V(\Phi^{-1}(1/\phi(t^{-1})))} \frac 1{ (\Phi(z)\phi(t^{-1}))^{1+d_2/\alpha_1}}.
\end{split}\end{equation}
On the other hand, changing the variable $s=r\phi(t^{-1})$ we have
\begin{equation}\label{e:I22n}\begin{split}
I_{2,2} & =  \frac{ c_7}{V(\Phi^{-1}(1/\phi(t^{-1})))}\\
&\quad \times
\int_{c_1 \Phi(z) \phi(t^{-1})}^\infty
\frac{V(\Phi^{-1}(1/\phi(t^{-1})))}
{V(\Phi^{-1}(s/\phi(t^{-1}) ))}\,
 s^{-1}\, \exp(-c_8t(\phi')^{-1}(t\phi(t^{-1})/s))\,ds \\
 &\le \frac{c_{12}}{V(\Phi^{-1}(1/\phi(t^{-1})))} \\
 &\quad\times\int_{c_1 \Phi(z) \phi(t^{-1})}^\infty
 s^{-1-d_1/\alpha_2} \exp(-c_8t(\phi')^{-1}(t\phi(t^{-1})/s))\,ds \\
&\le \frac{c_{13}}{V(\Phi^{-1}(1/\phi(t^{-1})))} \frac1{(\Phi(z) \phi(t^{-1}))^{d_1/\alpha_2}} \\
&\quad \times \exp(-c_{8}t(\phi')^{-1}(t\phi(t^{-1})/(\Phi(z) \phi(t^{-1}))))
\\
&\le \frac{c_{14}}{V(\Phi^{-1}(1/\phi(t^{-1})))}
\exp\Big(-c_{15}(\Phi(z) \phi(t^{-1}))^{1/(1-\beta_2)}\Big),
\end{split}\end{equation}
where we used \eqref{vd1} and \eqref{e:Phii} again
in the first inequality, and used
\eqref{e:04} and \eqref{e:03} in the last inequality.
From \eqref{e:I21}--\eqref{e:I22n} we  can choose a constant $c^*>0$ large enough such that for all
$t>0$ and $z\ge 0$ with
$\Phi(z)\phi(t^{-1}) >c^*$, it holds that
$$I_{2,2} \le 2^{-1} I_{2,1}.$$ Thus,
$$
p(t,x, y) \ge
  \frac {2^{-1} c_{10}}{\phi(t^{-1})V(z)\Phi(z)}.
 $$
Moreover, if $c^* \ge \Phi(z)\phi(t^{-1}) \ge 1$, we see from Lemma \ref{L:lower1m}  that
$$
p(t,x, y) \ge
   \frac {c_{16}}{\phi(t^{-1})V(z)\Phi(z)}.
$$
Therefore,  when  $\Phi(z)\phi(t^{-1}) \ge 1$,
$$p(t,x,y)\ge
  \frac {c_{17}}{\phi(t^{-1}) V(z)\, \Phi(z)}.
 $$
This completes the proof.
\end{proof}

\section{Local spatial motions} \label{S:5}

\subsection{Time derivative of heat kernel estimates for diffusion processes}

 In this section, we consider the diffusion case where the associated heat kernel $q(t,x,y)$ satisfies
\eqref{eq:fibie3}.
In the following, set
$$\bar q( t,x, r):= \frac{1}{V(x,\Phi^{-1}(t))}\exp\left(-m(t,r)\right),\quad t,r>0 \hbox{ and } x\in M.$$
Applying Lemma \ref{p:dd1} to $V(x,\cdot)$, $\Phi(\cdot)$ and $1/m(\cdot, r)$,
we may and do
assume that all $V(x,\cdot)$, $\Phi(\cdot)$ and $m(\cdot, r)$ are differentiable, that
$V(x,\cdot)$ and $\Phi(\cdot)$ satisfy the property like
\eqref{e:Phiinvesd},
 and  that $m(\cdot, r)$ satisfies
\begin{align}
\label{e:diffm}
m(t,r)\simeq -t \frac{\partial m(t, r)}{\partial t} \quad \text{ for all }t, r >0.
\end{align}
Then  similar to those in Lemma $\ref{L:gd}$, we have the following time derivative estimates for $\bar q( t,x, r)$ defined above.
\begin{lemma}\label{L:gdd}
Under all assumptions above, there exist constants $c_0, c_0^*>0$ such that
for all $t,r>0$ and $x\in M$,
\begin{equation}\label{e:02m1}
\left|\frac{\partial \bar q(t,x, r)}{\partial t}\right|\le \frac{c_0}{t V(x,\Phi^{-1}(t))}\exp\left(-c_0^*m(t,r)\right)=:\frac{c_0}{t}  {\bar q}^*(t,x, r) ,
\end{equation}
and that there exist
constants $c_1,c_2>0$, $c_*\in(0,1)$ and $c^*\in (1,\infty)$ such that for all $x\in M$,
\begin{align}\label{e:04m11}
\frac{\partial \bar q(t,x, r)}{\partial t} \le -c_1
 \frac{\bar q(t,x, r)}t \quad \text{if }\,\,\Phi(r) \le c_* t,
\end{align}
and
\begin{align}\label{e:04m21}
\frac{\partial \bar q(t,x, r)}{\partial t}\ge c_2 \frac{\bar  q(t,x, r)}t \quad \text{if }\,\,  \Phi(r) \ge c^* t.
\end{align}
\end{lemma}

\begin{proof}
Since
$$\frac{\partial \bar q(t,x,r)}{\partial t}=\bar q(t,x,r)\left(-\frac{\partial_r V(x,\Phi^{-1}(t))(\Phi^{-1}(t))'}{V(x,\Phi^{-1}(t))}-\partial_t m(t,r)\right),$$
 we have
  by \eqref{e:Phiinvesd} and \eqref{e:diffm}
\begin{align}\label{ppffee}
\frac{\bar q(t,x,r)}{t}\left(-c_1+c_2m(t,r)\right) \le \frac{\partial \bar q(t,x,r)}{\partial t}\le  \frac{\bar q(t,x,r)}{t}\left(-c_1^{-1}+c_2^{-1}m(t,r)\right).\end{align}
This along with the fact that $r e^{-r}\le 2 e^{-r/2}$ for all $r>0$ immediately yields \eqref{e:02m1}.

Note that,
by \eqref{e:mphisim1}, $m(\Phi(r),r)\simeq1$. If $\Phi(r) \le c_* t$, then by \eqref{e:scdf-1},
$$\frac{1}{m(t,r)}\simeq \frac{m(\Phi(r),r)}{m(t,r)}\ge c_3\left(\frac t{\Phi(r)}\right)^{1/(\alpha_2-1)}\ge c_3\left(\frac{1}{c_*}\right)^{1/(\alpha_2-1)}$$ and so
$$m(t,r)\le c_4 {c_*}^{1/(\alpha_2-1)}.$$ By this and \eqref{ppffee}, we can take $c_*>0$ small enough such that \eqref{e:04m11} is satisfied.

Similarly,  if $\Phi(r) \ge c^* t$ for $c^*>1$ large enough, then $m(t,r)\ge 2c_1/c_2$ and so
$$-c_1+c_2m(t,r)\ge c_1,$$
which combined with \eqref{ppffee} in turn gives us \eqref{e:04m21}. The proof is complete.
\end{proof}

\subsection{Two-sided estimates for $p(t,x,y)$}

\begin{proof}[Proof of Theorem $\ref{theorem:maindiff}$]
We will closely follow the approach of Theorem \ref{theorem:mainjump} but need to carry out some non-trivial modifications.
We fix $x \in M$ and, for simplicity, we again
denote $z=d(x, y)$, and write $\bar q(t,x,z)$,
$\bar q^*(t,x,z)$ and $V(x,r)$ as $\bar q(t,z)$, $\bar q^*(t,z)$ and $V(r)$, respectively.
The proof is divided into two parts again.

\medskip

\noindent{\bf Proof of the upper bound of $p(t,x,y)$. }\,\,
By \eqref{e:1.5} and \eqref{eq:fibie3}
we have
$$
p(t, x, y) \asymp \int_0^\infty  \bar q(r,z)\, d_r \bP(S_r\geq t) =I_1+I_2,
$$
where
$$
I_1:=\int_0^{2/\phi(t^{-1})}
\bar q(r, z)
 \,d_r \bP (S_r\ge t) \quad \text{and} \quad
I_2:=-\int_{2/\phi(t^{-1})}^\infty
\bar q(r, z)
\,d_r \bP (S_r\le  t). $$
Following the same arguments  as \eqref{e:07-1} and \eqref{e:07-2}, and using Proposition \ref{p:sub} and \eqref{e:02m1}, we have
\begin{equation*}\label{e:07m-1} \begin{split}
I_1
\le & c{\bar q}(2/\phi(t^{-1}),z) - \int_0^{2/\phi(t^{-1})} \bP (S_r\ge t)\,d_r {\bar q}(r,z)\\
\le &c {\bar q}(2/\phi(t^{-1}),z)+c\int_0^{2/\phi(t^{-1})} r\phi (t^{-1}) \cdot \frac{1}{r} \cdot {\bar q}^*(r,z)\,dr\\
=&\!\!:c {\bar q}(2/\phi(t^{-1}), z)+c I_{1,1}\end{split}
\end{equation*}
and \begin{align*}I_2
=&-  {\bar q}(r,z )\bP (S_r\le t) \big|_{2/\phi(t^{-1})}^\infty +\int_{2/\phi(t^{-1})}^\infty \bP (S_r\le t)\,d_r {\bar q}(r,z)\\
\le & c{\bar q}(2/\phi(t^{-1}),z)+ c\int_{2/\phi(t^{-1})}^\infty \exp(-c_1 t (\phi')^{-1}(t/r)) \cdot \frac 1 r \cdot {\bar q}^*(r,z)\,dr\\
=&\!\!:c{\bar q}(2/\phi(t^{-1}),z)+ c I_{2,1}. \end{align*}

{\bf (1-a)} Suppose that $\Phi(z)\phi(t^{-1})\le 1$. Then
\begin{align*}I_{1,1}&=\phi(t^{-1})\int_0^{\Phi(z)} \frac{1}{V(\Phi^{-1}(r))} \exp(-c_0^*m(r,z))\,dr\\
&\quad + \phi(t^{-1})\int_{\Phi(z)}^{2/\phi(t^{-1})} \frac{1}{V(\Phi^{-1}(r))}\,dr\\
&=: I_{1,1,1}+I_{1,1,2}.\end{align*}
According to \eqref{vd1}, \eqref{e:Phii} and \eqref{e:scdf-1}, we have
\begin{equation}\label{e:PF}\begin{split} I_{1,1,1}=&\phi(t^{-1})\int_0^{\Phi(z)} \frac{1}{V(\Phi^{-1}(r))} \exp(-c_0^*m(r,z))\,dr\\
=&\phi(t^{-1})\sum_{n=0}^\infty \int_{\Phi(z)/2^{n+1}}^{\Phi(z)/2^{n}}\frac{1}{V(\Phi^{-1}(r))} \exp(-c_0^*m(r,z))\,dr\\
\le& \phi(t^{-1})\sum_{n=0}^\infty \frac{\Phi(z)/2^{n+1}}{V(\Phi^{-1}(\Phi(z)/2^{n+1}))} \exp(-c_0^*m(\Phi(z)/2^{n},z))\\
\le& \frac{c\phi(t^{-1})\Phi(z)}{V(z)} \sum_{n=0}^\infty 2^{n((d_2/\alpha_1)-1)} \exp(-c_1m(\Phi(z),z) 2^{n/(\alpha_2-1)}) \\
\le& \frac{c\phi(t^{-1})\Phi(z)}{V(z)}, \end{split}\end{equation} where in the last inequality we used the fact that $m(\Phi(z),z)\simeq 1$. This estimate along with \eqref{e:com1} and \eqref{e:rrffee1} yields that
$$I_1\le c I_{1,1,2}+cI_{1,1}\le  c_1 I_{1,1}\le  c_2 \phi(t^{-1}) \int_{\Phi(z)}^{2/\phi(t^{-1})} \frac 1{V(\Phi^{-1}(r))} \,dr.$$

{\bf (1-b)} Suppose that $\Phi(z)\phi(t^{-1})\ge 1$.  Then also by
\eqref{vd1}, \eqref{e:Phii} and \eqref{e:scdf-1}, we have
\begin{align*}I_{1,1}= &\phi(t^{-1})\int_{0}^{2/\phi(t^{-1})} \frac{1}{V(\Phi^{-1}(r))} \exp(-c_0^*m(r,z))\,dr\\
=&\phi(t^{-1})\sum_{n=0}^\infty \int_{2/(2^{n+1}\phi(t^{-1}))}^{2/(2^{n}\phi(t^{-1}))}\frac{1}{V(\Phi^{-1}(r))} \exp(-c_0^*m(r,z))\,dr\\
\le& \phi(t^{-1})\sum_{n=0}^\infty \frac{2/(2^{n+1}\phi(t^{-1}))}{V(\Phi^{-1}(2/(2^{n+1}\phi(t^{-1}))))} \exp(-c_0^*m(2/(2^{n}\phi(t^{-1})),z))\\
\le& \frac{c}{V(\Phi^{-1}(1/\phi(t^{-1})))} \sum_{n=0}^\infty 2^{n((d_2/\alpha_1)-1)} \exp\left(-c_1m(1/\phi(t^{-1}),z) 2^{n/(\alpha_2-1)}\right) \\
\le& \frac{c}{V(\Phi^{-1}(1/\phi(t^{-1})))} \exp\left(-c_2m(1/\phi(t^{-1}),z)\right), \end{align*} which implies that
$$I_1\le \frac{c}{V(\Phi^{-1}(1/\phi(t^{-1})))} \exp\left(-c_0m(1/\phi(t^{-1}),z)\right).$$

{\bf (2-a)} Suppose that $\Phi(z)\phi(t^{-1})\le 1$.
Then by
the argument of \eqref{com1-100},
\begin{align*}I_{2,1}&=\int_{2/\phi(t^{-1})}^\infty \exp(-c_1 t (\phi')^{-1}(t/r)) \cdot r^{-1}\cdot {\bar q}^*(r,z)\,dr\\
&\le c  \int_{2/\phi(t^{-1})}^\infty \exp(-c_1 t (\phi')^{-1}(t/r)) \cdot r^{-1} \cdot \frac1 {V(  \Phi^{-1}(r))} \,dr\le \frac{c}{V(\Phi^{-1}(1/\phi(t^{-1})))},\end{align*}
hence
 $$I_2\le  \frac{c}{V(\Phi^{-1}(1/\phi(t^{-1})))}.$$

 {\bf (2-b)} We now consider the case that $\Phi(z)\phi(t^{-1})\ge 1$, which is more complex and difficult than the previous case.

To get the estimate for $I_{2,1}$, we need to consider the
following two functions inside the exponential terms of ${\bar q}^*(r,z)$ and the estimates of $\bP (S_r\le t)$ respectively:
\begin{equation}\label{o:fun}
G_1(r)=t(\phi')^{-1}(t/r) ~~\mbox{ and }~~G_2(r)=m(r,z)\end{equation}
for all $r>0$ and fixed $z,t >0$.
Note that, by
\eqref{e:03}, \eqref{e:scdf-1} and the facts that $\phi'$ and $m(\cdot,z)$ are non-increasing on $(0,\infty)$,
$G_1(r)$ is
a non-decreasing
function on $(0,\infty)$ such that $G_1(0)=0$ and $G_1(\infty)=\infty$, and $G_2(r)$ is a non-increasing function on $(0,\infty)$ such that $G_2(0)=\infty$ and
$G_2(\infty)=0$.
Thus, there is a unique
$r_0=r_0(z,t)\in (0,\infty)$
such that
$G_1(r_0)=G_2(r_0)$, $G_1(r)\ge G_2(r)$ when $r\ge r_0$, and $G_1(r)\le G_2(r)$ when $r\le r_0$.

On the other hand, when $\Phi(z)\phi(t^{-1})\ge 1$,
by \eqref{e:04}, \eqref{e:mphisim1} and the fact that
$m(\cdot, z)$ is non-increasing on $(0,\infty)$,
\begin{align*}
G_1(1/\phi(t^{-1}))&=t(\phi')^{-1}(t\phi(t^{-1})) \simeq t(\phi')^{-1}(\phi'(t^{-1}))=1\\
&\le c_1m(\Phi(z),z)\le c_1m(1/\phi(t^{-1}), z)=c_1 G_2(1/\phi(t^{-1}))\end{align*}  and
\begin{align*} G_1(\Phi(z))\ge& G_1(1/\phi(t^{-1}))\simeq t(\phi')^{-1}(\phi'(t^{-1}))=1\\
\ge & c_2m(\Phi(z), z)= c_2 G_2(\Phi(z))
\end{align*}
where constants $c_1, c_2$ are independent of $t$ and $z$.
Hence there are constants $c_3,c_4>0$ independent of $t$ and $z$ such that
\begin{equation}\label{e:ssddee} \frac{2}{\phi(t^{-1})} \le c_3 r_0\le c_4 \Phi(z).\end{equation}

Combining all the estimates above, we find that
\begin{equation*}\begin{split}I_{2,1}&=\int_{2/\phi(t^{-1})}^\infty \exp(-c_0 t (\phi')^{-1}(t/r)) \cdot r^{-1}\cdot {\bar q}^*(r,z)\,dr\\
&\le
\frac{1}{ V(\Phi^{-1}(2/\phi(t^{-1})))}\int_{2/\phi(t^{-1})}^\infty \frac{1}{ r}\cdot \exp(-c_0 t (\phi')^{-1}(t/r))\cdot \exp(-c_0^*m(r,z))\,dr\\
&\le  \frac{c_5}{V(\Phi^{-1}(1/\phi(t^{-1})))}\int_{2/\phi(t^{-1})}^{c_3r_0} \frac 1 r \cdot \exp(-c_0^*m(r,z))\,dr\\
&\quad +\frac{c_5}{V(\Phi^{-1}(1/\phi(t^{-1})))}\int_{c_3r_0}^\infty \frac{1}{r}\cdot \exp(-c_0 t (\phi')^{-1}(t/r))\,dr \\
&=:\frac{c_5}{V(\Phi^{-1}(1/\phi(t^{-1})))} I_{2,1,1}+\frac{c_5}{V(\Phi^{-1}(1/\phi(t^{-1})))} I_{2,1,2}.   \end{split}\end{equation*}

According to \eqref{e:scdf-1},
\begin{align*}I_{2,1,1}&\le \int_0^{c_3r_0}\frac{1}{r}\cdot \exp(-c_0^*m(r,z))\,dr
=\sum_{n=0}^\infty \int_{c_3r_0/(2^{n+1})}^{c_3r_0/(2^n)} \frac 1 r \cdot \exp(-c_0^*m(r,z))\,dr\\
&\le c_6 \sum_{n=0}^\infty \exp(-c_0^*m(c_3r_0/2^n,z))\le c_6\sum_{n=0}^\infty \exp\left(-c_7m(r_0,z) 2^{n/(\alpha_2-1)}\right)\\
&\le  c_6 \exp(-c_8 G_2(r_0)). \end{align*}
On the other hand, by \eqref{e:031},
\begin{align*}I_{2,1,2}&=\sum_{n=0}^\infty \int_{2^nc_3r_0}^{2^{n+1}c_3r_0} \frac{1}{r}\cdot \exp(-c_0 t (\phi')^{-1}(t/r))\,dr\\
&\le c_9\sum_{n=0}^\infty \exp(-c_0 t (\phi')^{-1}(t/(2^nc_3r_0)))\le c_9 \sum_{n=0}^\infty \exp(-c_{10}t (\phi')^{-1}(t/r_0) 2^{n(1-\beta_2)}) \\
&\le c_9 \exp(-c_{11} t (\phi')^{-1}(t/r_0))=c_9 \exp(-c_{11}G_1(r_0))= c_9 \exp(-c_{11}G_2(r_0)). \end{align*}
Putting these estimates together, we have
$$I_{2,1}\le \frac{c_{12}}{V(\Phi^{-1}(1/\phi(t^{-1})))}\exp(-c_{13}G_2(r_0)).$$ Since $G_2(r_0)=G_1(r_0)\le c_{14} G_1(1/\phi(t^{-1}))$ (thanks to \eqref{e:ssddee}), we obtain
\begin{equation}\label{e:llff} \begin{split} I_{2}\le& \frac{c_{14}}{V(\Phi^{-1}(1/\phi(t^{-1})))}\exp(-c_{15}G_2(r_0)).\end{split}\end{equation}

\ \

Next, we rewrite the exponential term in the right hand side of \eqref{e:llff}. By the fact that
$m(r_0,z)=G_2(r_0)=G_1(r_0)=t(\phi')^{-1}(t/r_0)$ and the definition of $m(r_0,z)$, we have
$$\frac{r_0}{t(\phi')^{-1}(t/r_0)}\simeq \Phi\left( \frac{z}{t(\phi')^{-1}(t/r_0)}\right).$$   Let $s_0=(\phi')^{-1}(t/r_0)$. Then $t/r_0=\phi'(s_0)$ and, by \eqref{e:04},
\begin{equation}\label{e:llfgg}\frac{1}{\phi((ts_0)/t)}= \frac{1}{\phi(s_0)}\simeq \frac{1}{\phi'(s_0)s_0}\simeq \Phi\left( \frac{z}{ts_0}\right).\end{equation} Thus,
$G_2(r_0)=G_1(r_0)= t(\phi')^{-1}(t/r_0)= t(\phi')^{-1}(t/(t/\phi'(s_0))) = ts_0.$ This together with  \eqref{e:llff} and \eqref{e:llfgg} yields that
$$ I_2 \le \frac{c_{16}}{V(x, \Phi^{-1}(1/\phi(t^{-1})))} \exp(-c_{17}n(t,z)),$$ where
$n=n(t,z)$ satisfies
$$\frac{1}{\phi(n/t)} \simeq \Phi\left(\frac{z}{n}\right).$$

\medskip
Combining all the estimates above,
we get the desired upper bounded estimates for $p(t,x,y).$

\medskip

\noindent
{\bf Proof of the lower bound of $p(t,x,y)$.}\,\,
{\bf(1)} Suppose that $\Phi(z)\phi(t^{-1})\le 1$.
In this case, the proof is almost the same as the jump case except that one uses Lemma \ref{L:gdd}
in place of Lemma \ref{L:gd}.
Nevertheless for reader's convenience, we present a proof here.
By \eqref{e:02m1} and \eqref{e:04m11}, we can find a constant $c_1>1$ such that when $c_1 \Phi(z)\le r,$
\begin{align}\label{e:barut1}
 \frac{\partial {\bar q}(r,z) }{\partial r} \le -\frac{c_2}{r} {\bar q}(r,z);
 \end{align}
  and when  $ 0< r\le c_1 \Phi(z)$,
\begin{align}\label{e:barut2}  \frac{\partial {\bar q}(r,z) }{\partial r} \le \frac{c_3}{r}{\bar q}^*(r,z). \end{align}
Using  the fact that
 $$\int_{r_0}^\infty  {\bar q}(r,z) \,d_r \bP (S_r\le  t)   \le 0,\quad r_0>0,$$ following the arguments of \eqref{e:07} and \eqref{e:07-1}, and applying \eqref{e:barut1} and \eqref{e:barut2}, we find that
\begin{align*}
p(t, x, y ) &\ge c_4\phi(t^{-1}) \int_{c_1 \Phi(z)}^{2c_1/\phi(t^{-1})} \frac{1}{V(\Phi^{-1}(r))}\,dr\\
&\quad -{c_5 \phi(t^{-1})}\int^{c_1 \Phi(z)}_0    \frac{1}{V(\Phi^{-1}(r))}\exp(-c_6m(r,z))\,dr \\
&=: I_{1,1}-I_{1,2}.
\end{align*}
According to \eqref{e:PF},
\begin{align*}I_{1,2}
&\le  \frac{c_7\phi(t^{-1})\Phi(z) }{V(z)},\end{align*} and so
\begin{align*} p(t,x,y)
\ge &c_4\int_{c_1 \Phi(z) \phi(t^{-1})}^{2c_1} \frac 1{V(\Phi^{-1}( s/\phi(t^{-1})))}\, ds - \frac{c_7\phi(t^{-1})\Phi(z) }{V(z)}.
 \end{align*}
Therefore, combining this  estimates and Lemma \ref{L:lower1m}, we obtain
\begin{align*} p(t,x,y)
\ge&c_8\int_{\Phi(z) \phi(t^{-1})}^{2} \frac 1{V(\Phi^{-1}( s/\phi(t^{-1})))}\, ds.
 \end{align*}
 See the end of part {\bf(1)} in the proof of the lower bound estimates of  $p(t,x,y)$ in Subsection \ref{s:4.2}.

{\bf(2)} Suppose that $\Phi(z)\phi(t^{-1})\ge 1$. By \eqref{e:04m21} and \eqref{e:02m1}, we can find a constant $0<c_1<1$ such that when $ r \le c_1 \Phi(z),$
\begin{equation}\label{rkl}  \frac{\partial {\bar q}(r,z) }{\partial r} \ge  \frac{c_2}{r}{\bar q}(r,z),
\end{equation}
 while for $ r\ge c_1 \Phi(z)$,
\begin{equation}\label{rkl0}
\frac{\partial {\bar q}(r,z) }{\partial r}
\ge  - \frac{c_3}{r}{\bar q}^*(r,z)\end{equation}
Using \eqref{e:04}, \eqref{rkl} and \eqref{rkl0}, following the arguments of
\eqref{e:07} and \eqref{e:07-2}, and noting that $\int^{c_1 \Phi(z)}_{a}=\int^\infty_{a}-\int^\infty_{c_1 \Phi(z)}$,
we find that for any $a \in (0, c_1\Phi(z)]$,
\begin{align*}p(t,x, y)\ge & c_4\int^{c_1 \Phi(z)}_{a}  \frac{1}{V(\Phi^{-1}(r))} \cdot \frac{1}{r} \cdot \exp(-c_5t (\phi')^{-1}(t/r)) \cdot \exp(-c_5m(r,z)) \,dr\\
&-c_6\int_{c_1 \Phi(z)}^\infty  \frac{1}{V(\Phi^{-1}(r))}\cdot \frac{1}{r} \cdot \exp(-c_7t (\phi')^{-1}(t/r)) \cdot \exp(-c_7m(r,z))\,dr
\\
\ge & c_4\int^{\infty}_{a}  \frac{1}{V(\Phi^{-1}(r))} \cdot \frac{1}{r} \cdot \exp(-c_5t (\phi')^{-1}(t/r)) \cdot \exp(-c_5m(r,z))\,dr\\
&-(c_4+c_6)\int_{c_1 \Phi(z)}^\infty  \frac{1}{V(\Phi^{-1}(r))}\cdot \frac{1}{r} \cdot \exp(-c_7t (\phi')^{-1}(t/r)) \,dr
\\=:&c_4I_{2,1}(a)-(c_4+c_6)I_{2,2}.
\end{align*}

By \eqref{vd1}, \eqref{e:Phii} and \eqref{e:031},
\begin{align*}I_{2,2}&
=\sum_{n=0}^\infty\int_{c_1 2^n\Phi(z)}^{c_1 2^{n+1}\Phi(z)}   \frac{1}{V(\Phi^{-1}(r))}\cdot \frac{1}{r} \cdot \exp(-c_7t(\phi')^{-1}(t/r)) \,dr\\
&\le \sum_{n=0}^\infty \frac{1}{V(\Phi^{-1}(c_1 2^n\Phi(z)))}\exp(-c_7t(\phi')^{-1}(t/(c_1 2^n\Phi(z))))\\
&\le \frac{c_8}{V(z)}\sum_{n=0}^\infty 2^{-nd_1/\alpha_2} \exp(-c_9 2^{n(1-\beta_2)}
t(\phi')^{-1}(t/\Phi(z))  )\\
&\le  \frac{c_{10}}{V(z)}\exp(-c_{11} t(\phi')^{-1}(t/\Phi(z))). \end{align*}

Similar argument as above yields that
\begin{equation}
\label{e:newlow1}\begin{split} &\int^{\infty}_{a}  \frac{1}{V(\Phi^{-1}(r))} \cdot \frac{1}{r} \cdot \exp(-2c_5t (\phi')^{-1}(t/r)) \,dr\\
&=\sum_{n=0}^\infty\int_{a2^n}^{a 2^{n+1}}   \frac{1}{V(\Phi^{-1}(r))}\cdot \frac{1}{r} \cdot \exp(-2c_5t(\phi')^{-1}(t/r)) \,dr\\
&\ge \frac{1}{2} \sum_{n=0}^\infty \frac{1}{V(\Phi^{-1}(a 2^{n+1}))}\exp(-2c_5t(\phi')^{-1}(t/(a 2^{n+1})))\\
&\ge \frac{c_{12}}{V(\Phi^{-1}(a))}\sum_{n=0}^\infty 2^{-nd_2/\alpha_1}
\exp(-c_{13}2^{n(1-\beta_1)} t(\phi')^{-1}(t/a))
\\
&\ge  \frac{c_{14} }{V(\Phi^{-1}(a))}\exp(-c_{15}  t(\phi')^{-1}(t/a)),
\end{split}\end{equation}
where all the constants $c_k$'s are independent of $a$.

Without loss of generality we assume $c_{11}<2c_5<c_{15}$, and let
\begin{equation}\label{o:funs2}
G_1^*(r)=c_{11} t(\phi')^{-1}(t/r) ~~\mbox{ and }~~G_2^*(r)=2c_{15} m(r,z)\end{equation}
for all $r>0$ and fixed $z,t >0$. We let $r_0=r_0(t,z)>0$ be the unique constant such that $G_1^*(r_0)=G_2^*(r_0)$. Note that  $G_1^*(r)\ge G_2^*(r)$ for all $r\ge r_0$. In particular,
$$t (\phi')^{-1}(t/r)\ge m(r,z),\quad r\ge r_0,$$ and so
\begin{equation}\label{e:newlow2}
\exp(-c_5t (\phi')^{-1}(t/r)) \cdot \exp(-c_5m(r,z)) \ge \exp(-2c_5t (\phi')^{-1}(t/r)),\quad r \ge r_0.
\end{equation}
By \eqref{e:04}
and \eqref{e:03},
we can choose
$C_0>1$ large such that
\begin{equation}\label{e:ssddee1}\begin{split}
 G_1^*(c_1C_0/\phi(t^{-1})) =& c_{11} t(\phi')^{-1}(t\phi(t^{-1})/(c_1C_0))\\
 \ge& c_{11} t(\phi')^{-1}(c_{17}\phi'(t^{-1})/(c_1C_0))\\
 >&(4 c_{16}c_{15})  \vee (2\log(2c_{10}(c_4+c_6)/(c_4 c_{14}))),
\end{split}\end{equation}
where in the last inequality $c_{16}>0$ satisfies that
$m(c_1 \Phi(z),z)\le c_{16}$
(due to \eqref{e:scdf-1} and \eqref{e:mphisim1}).
Then since $G_1^*$ is non-decreasing,  if  $\Phi(z)\phi(t^{-1})\ge C_0$
$$
G_1^*(c_1\Phi(z))\ge G_1^*(c_1C_0/\phi(t^{-1}))\ge 4 c_{16}c_{15}
\ge   4c_{15} m(c_1 \Phi(z), z)= 2 G_2^*(c_1 \Phi(z)),$$
which, in particular, implies that  $c_1\Phi(z)\ge r_0$.
Thus, we can take $a=r_0$ in \eqref{e:newlow1} and find that for $\Phi(z)\phi(t^{-1})\ge C_0$,
\begin{align*}I_{2,1}(r_0)
&\ge
 \frac{c_{14} }{V(\Phi^{-1}(r_0))}\exp(-2^{-1} G_2^*(r_0))=  \frac{c_{14} }{V(\Phi^{-1}(r_0))}\exp(-2^{-1} G_1^*(r_0)).\end{align*}
Therefore, combining
all the inequalities above, we obtain that for  $\Phi(z)\phi(t^{-1})\ge C_0$,
 \begin{align*}
 p(t,x,y) \ge \frac{c_4 c_{14} }{V(\Phi^{-1}(r_0))}\exp(-2^{-1} G_1^*(r_0)) -
  \frac{c_{10}(c_4+c_6)}{V(z)}\exp(- G_1^*(\Phi(z))).
 \end{align*}
 By
 the fact that $G_1^*$ is non-decreasing and \eqref{e:ssddee1},
 for  $\Phi(z)\phi(t^{-1})\ge C_0$,  \begin{align*}
\exp( 2^{-1} G_1^*(c_1\Phi(z))  ) \ge \exp( 2^{-1}G_1^*(c_1C_0/\phi(t^{-1})))\ge
2c_{10}(c_4+c_6)/(c_4 c_{14}),
 \end{align*}
so that,  using
again the fact that $G_1^*$ is a non-decreasing function, we have
 \begin{align*}
\frac{c_4 c_{14} }{V(\Phi^{-1}(r_0))}\exp(-2^{-1} G_1^*(r_0))  &\ge
\frac{c_4 c_{14} }{V(z)}\exp(-2^{-1} G_1^*(c_1\Phi(z))) \\
 &
=\frac{c_4 c_{14} }{V(z)}\exp(2^{-1} G_1^*(c_1\Phi(z))) \exp(- G_1^*(\Phi(z))) \\
&\ge
\frac{2c_{10}(c_4+c_6)}{V(z)}\exp(- G_1^*(\Phi(z))) .
 \end{align*}
  Thus  for  $\Phi(z)\phi(t^{-1})\ge C_0$
\begin{equation}\label{eehh} p(t,x, y)\ge \frac{2^{-1} c_4 c_{14} }{V(\Phi^{-1}(r_0))}\exp(-2^{-1} G_1^*(r_0)).\end{equation}
 By \eqref{e:04},
$$G_1^*\left({1}/{\phi(t^{-1})}\right)=c_{11}t(\phi')^{-1}(t\phi(t^{-1}))\simeq 1.$$
Using this,
\eqref{vd1}, \eqref{e:03} and \eqref{e:ssddee}, we have
\begin{align*}
& \frac{1}{V(\Phi^{-1}(r_0))}\exp(-2^{-1} G_1^*(r_0)) \\
&\ge \frac{1}{V(\Phi^{-1}(1/\phi(t^{-1})))}\exp(-G_1^*(r_0)) \left[\frac{V(\Phi^{-1}(1/\phi(t^{-1})))}{V(\Phi^{-1}(r_0))}\exp\left(\frac{c_{18}G_1^*(r_0)}{G_1^*({1}/{\phi(t^{-1})})} \right) \right]\\
&\ge \frac{c_{19}}{V(\Phi^{-1}(1/\phi(t^{-1})))}\exp(-G_1^*(r_0))
\left[(r_0\phi(t^{-1}))^{-d_2}\exp(c_{20}(r_0\phi(t^{-1}))^{1/(1-\beta_1)})\right]\\
& \ge\frac{c_{21}}{V(\Phi^{-1}((1/\phi(t^{-1})))}\exp(- G_1^*(r_0)),
\end{align*}
where in the last inequality we used the fact that
$\inf_{r>0} r^{-d} \exp(c_{20} r^{1/(1-\beta)})>0$.

Combining the inequality above with \eqref{eehh}, we obtain that  for  $\Phi(z)\phi(t^{-1})\ge C_0$
$$p(t,x,y)\ge
 \frac{c_{22}}{V(\Phi^{-1}(1/\phi(t^{-1}))))}\exp(-G_1^*(r_0)).
 $$ Furthermore,
by the same argument as
that for
the expression of $G_1(r_0)$
at the end of part {\bf(2-b)} in the proof of upper bound for $p(t,x,y)$, we
  arrive at that for  $\Phi(z)\phi(t^{-1})\ge C_0$
  $$ p(t,x,y)\ge \frac{c_{22}}{V(\Phi^{-1}(1/\phi(t^{-1}))))} \exp(-c_{23}n(t,z)),$$ where
$$\frac{1}{\phi(n/t)} \simeq \Phi\left(\frac{z}{n}\right).$$
Combining this with Lemma \ref{L:lower1m}, we finish the proof.
\end{proof}

 \  \

At the end of the section, we give a corollary of Theorems \ref{theorem:mainjump} and \ref{theorem:maindiff}, which is concerned with explicit forms for the estimate \eqref{e:llffoo}.

\begin{corollary}\label{T:mix2}
Suppose that the fundamental solution $p(t,x,y)$ is  given by \eqref{e:1.5},
where the heat kernel $q(t, x, y)$ of $X$ (or equivalently, of $\sL$)  satisfies condition either \eqref{e:hkmequi} or \eqref{eq:fibie3}.
Then the following  hold.
\begin{itemize}
\item[(i)] If  $d_2<\alpha_1$ and $\Phi(d(x,y))\phi(t^{-1}) \le 1$, then
\begin{align*} p(t,x,y)
\simeq
\frac{1}{V(x,\Phi^{-1}(1/\phi(t^{-1})))}.
\end{align*}

\item[(ii)] If $d_1>\alpha_2$ and $\Phi(d(x,y))\phi(t^{-1})
  \le 1$,
then
\begin{align*} p(t,x,y)
\simeq
  \frac{\Phi(d(x,y))  \phi(t^{-1}) } { V(x,d(x,y))}.
 \end{align*}

 \item[(iii)] If $d_1=d_2=\alpha_1=\alpha_2$ and $\Phi(d(x,y))\phi(t^{-1})
  \le 1$,
then
  $$p(t,x,y)
  \simeq
\frac{1}{V(x, \Phi^{-1}(1/\phi(t^{-1})))}\log\left( \frac{2}{\Phi(d(x,y)) \phi(t^{-1})}\right).$$ \end{itemize}
 \end{corollary}

\begin{proof} We use the same notation as
in the proof of Theorem \ref{theorem:mainjump}.
We have already observed in \eqref{e:rrffee1} and  \eqref{e:com1} that, when $\Phi(z)\phi(t^{-1}) \le 1$,
\begin{align*}
\int_{ \Phi(z) \phi(t^{-1})}^{2 } \frac 1{V(\Phi^{-1}( s/\phi(t^{-1})))}\,ds=&\phi(t^{-1}) \int_{\Phi(z)}^{2/\phi(t^{-1})} \frac 1{V(\Phi^{-1}(r))}\, dr\\
\ge& \frac{c}{V(\Phi^{-1}(1/\phi(t^{-1})))}
\end{align*}
and $$ \int_{ \Phi(z) \phi(t^{-1})}^{2 } \frac 1{V(\Phi^{-1}( s/\phi(t^{-1})))}\,ds=\phi(t^{-1}) \int_{\Phi(z)}^{2/\phi(t^{-1})} \frac 1{V(\Phi^{-1}(r))}\, dr\ge \frac{c{\Phi(z) } \phi(t^{-1})}{V( z)} .$$

(i) Suppose that $d_2<\alpha_1$. Then for $\Phi(z)\phi(t^{-1}) \le 1$,
\begin{equation}\label{iirr}\begin{split} &\int_{ \Phi(z) \phi(t^{-1})}^{2 } \frac 1{V(\Phi^{-1}( s/\phi(t^{-1})))}\,ds\\
&= \frac{1}{V(\Phi^{-1}(1/\phi(t^{-1})))}\int_{ \Phi(z) \phi(t^{-1})}^{2 }
\frac{V(\Phi^{-1}(1/\phi(t^{-1})))}{V(\Phi^{-1}( s/\phi(t^{-1})))}
\,ds\\
&\le \frac{c}{V(\Phi^{-1}(1/\phi(t^{-1})))}\int_{ \Phi(z) \phi(t^{-1})}^{2 } s^{-d_2/\alpha_1} \,ds\\
&\le \frac{c}{V(\Phi^{-1}(1/\phi(t^{-1})))}\int_{ 0}^{2 } s^{-d_2/\alpha_1} \,ds\le \frac{c}{V(\Phi^{-1}(1/\phi(t^{-1})))},
\end{split}\end{equation} where in the first inequality we used \eqref{vd1} and \eqref{e:Phii}.

(ii) Suppose that $d_1>\alpha_2$. Then using \eqref{vd1} and \eqref{e:Phii} again, we have that for $\Phi(z)\phi(t^{-1})
  \le 1$,
\begin{equation*}\begin{split} &\int_{ \Phi(z) \phi(t^{-1})}^{2}
\frac 1{V(\Phi^{-1}( s/\phi(t^{-1})))}
\, ds\\
&= \frac 1{V(z)}\int_{ \Phi(z) \phi(t^{-1})}^{2 }
\frac{V(\Phi^{-1}(\Phi(z) ))}{V(\Phi^{-1}( s/\phi(t^{-1})))}
\,ds \le \frac c{V(z)}\int_{ \Phi(z) \phi(t^{-1})}^{2 }
\left(\frac{\Phi(z) }{ s/\phi(t^{-1})} \right)^{d_1/\alpha_2}
\,ds\\
& \le \frac c{V(z)}({\Phi(z) }{ \phi(t^{-1})})^{d_1/\alpha_2}\int_{ \Phi(z) \phi(t^{-1})}^{\infty }
s^{-d_1/\alpha_2}
\,ds\le  \frac{ c \Phi(z)  \phi(t^{-1}) }{V(z)}.
\end{split}\end{equation*}

(iii) When $d_1=d_2=\alpha_1=\alpha_2$, by the argument of \eqref{iirr}, we have for $\Phi(z)\phi(t^{-1}) \le 1$,
\begin{equation*}\begin{split} &\int_{ \Phi(z) \phi(t^{-1})}^{2 } \frac 1{V(\Phi^{-1}( s/\phi(t^{-1})))} \,ds\\
&= \frac 1{V(\Phi^{-1}(1/\phi(t^{-1})))}\int_{ \Phi(z) \phi(t^{-1})}^{2 }
\frac{V(\Phi^{-1}(1/\phi(t^{-1})))}{V(\Phi^{-1}( s/\phi(t^{-1})))}
\,ds\\
&\simeq \frac 1{V(\Phi^{-1}(1/\phi(t^{-1})))}\int_{ \Phi(z) \phi(t^{-1})}^{2 } s^{-1} \,ds= \frac{c}{V(\Phi^{-1}(1/\phi(t^{-1})))}\log\left( \frac{2}{\phi(t^{-1})\Phi(z)}\right).
\end{split}\end{equation*}

Therefore, the desired assertion now follows from all the estimates above.  \end{proof}

\begin{proof}[Proof of Theorem $\ref{T:1.2}$] The conclusion for the case that $d(x,y)\phi(t^{-1})\le 1$ immediately follows from Theorem \ref{theorem:mainjump} (i), Theorem \ref{theorem:maindiff} (i) and Corollary \ref{T:mix2}. When $d(x,y)\phi(t^{-1})\ge 1$, the assertion
for pure jump type Dirichlet form  $(\sE, \sF)$ is a direct consequence of Theorem \ref{theorem:mainjump} (ii); for
the case that $(\sE, \sF)$ is  local, according to Theorem \ref{theorem:maindiff} (ii), $n:=n(t,d(x,y))$ is now determined by
$$\frac{1}{\phi(n/t)}\simeq \left(\frac{d(x,y)}{n}\right)^\alpha,\quad t>0,x,y\in M.$$ This is,
$$\bar \phi_\alpha(n/t)=\frac{(n/t)^\alpha}{\phi(n/t)}
\simeq \left(\frac{d(x,y)}{t}\right)^\alpha,\quad t>0,x,y\in M.$$ Then we can prove the desired assertion. \end{proof}

\noindent \textbf{Acknowledgements.}
The research of Panki Kim is supported by
 the National Research Foundation of Korea (NRF) grant funded by the Korea government (MSIP)
(No.\ 2016R1E1A1A01941893).\ The research of Takashi Kumagai is supported
by the Grant-in-Aid for Scientific Research (A)
25247007 and 17H01093,
Japan.\ The research of Jian Wang is supported by National
Natural Science Foundation of China (No.\ 11522106), the JSPS postdoctoral fellowship
(26$\cdot$04021), Fok Ying Tung
Education Foundation (No.\ 151002), National Science Foundation of
Fujian Province (No.\ 2015J01003), the Program for Probability and Statistics: Theory and Application (No.\ IRTL1704), and Fujian Provincial
Key Laboratory of Mathematical Analysis and its Applications
(FJKLMAA).

\vskip 0.2truein

{\footnotesize {\bf Zhen-Qing Chen}

Department of Mathematics, University of Washington, Seattle,
WA 98195, USA

E-mail: zqchen@uw.edu

\bigskip

{\footnotesize {\bf Panki Kim}

Department of Mathematical Sciences,
Seoul National University,
Building 27, 1 Gwanak-ro,

Gwanak-gu,
Seoul 08826, Republic of Korea

E-mail: pkim@snu.ac.kr

\bigskip

{\bf Takashi Kumagai}

Research Institute for Mathematical Sciences,
Kyoto University, Kyoto 606-8502, Japan

E-mail: kumagai@kurims.kyoto-u.ac.jp

\bigskip

{\bf Jian Wang}

College of Mathematics and Informatics \& Fujian Key Laboratory of Mathematical

Analysis and Applications (FJKLMAA), Fujian Normal University, 350007 Fuzhou,
P.R. China.

E-mail: jianwang@fjnu.edu.cn
}

\end{document}